\pdfoutput=1

\documentclass[reqno]{amsart}

\usepackage{amsmath}
\usepackage{amssymb}
\usepackage{amsfonts}
\usepackage{amsthm}
\makeatletter
\def\thm@space@setup{\thm@preskip=4pt \thm@postskip=3pt}
\makeatother

\usepackage[foot]{amsaddr}
\usepackage{mathtools}

\usepackage[utf8]{inputenc}
\usepackage[T1]{fontenc}

\usepackage[sf,mono=false]{libertine}

\usepackage{inconsolata}

\usepackage{dsfont}

\usepackage[%
cal=cm,
]
{mathalfa}

\usepackage[dvipsnames,svgnames]{xcolor}
\colorlet{MyBlue}{DodgerBlue!75!Black}
\colorlet{MyGreen}{DarkGreen!85!Black}

\usepackage{microtype}							

\usepackage[font=footnotesize,labelfont=bf]{caption}
\usepackage{subfigure}
\usepackage{tikz}
\usetikzlibrary{arrows,calc,patterns}

\usepackage{acronym}
\usepackage{latexsym}
\usepackage{paralist}
\usepackage{xspace}

\usepackage[numbers,sort&compress]{natbib}

\usepackage{hyperref}
\hypersetup{
final,
colorlinks=true,
linktocpage=true,
pdfstartview=FitH,
breaklinks=true,
pdfpagemode=UseNone,
pageanchor=true,
pdfpagemode=UseOutlines,
plainpages=false,
bookmarksnumbered,
bookmarksopen=false,
bookmarksopenlevel=1,
hypertexnames=true,
pdfhighlight=/O,
urlcolor=Maroon,linkcolor=MyBlue!60!black,citecolor=DarkGreen!70!black, 
pdftitle={},
pdfauthor={},
pdfsubject={},
pdfkeywords={},
pdfcreator={pdfLaTeX},
pdfproducer={LaTeX with hyperref}
}

\usepackage[sort&compress,capitalize,nameinlink]{cleveref}
\crefname{assumption}{Assumption}{Assumptions}
\crefname{proofstep}{Step}{Steps}

\DeclarePairedDelimiter{\braces}{\{}{\}}					
\DeclarePairedDelimiter{\bracks}{[}{]}					

\DeclarePairedDelimiter{\abs}{\lvert}{\rvert}						
\DeclarePairedDelimiter{\norm}{\lVert}{\rVert}					
\DeclarePairedDelimiterXPP{\dnorm}[1]{}{\lVert}{\rVert}{_{\ast}}{#1}	

\DeclarePairedDelimiterX{\braket}[2]{\langle}{\rangle}{#1,#2}		
\DeclarePairedDelimiterX{\product}[2]{\langle}{\rangle}{#1,#2}		

\DeclarePairedDelimiterX{\setdef}[2]{\{}{\}}{#1:#2}					
\DeclarePairedDelimiterXPP{\exclude}[1]{\mathopen{}\setminus}{\{}{\}}{}{#1}

\newcommand{\eg}{e.g.,\xspace}						
\newcommand{\ie}{i.e.,\xspace}						
\newcommand{\textpar}[1]{\textup(#1\textup)}					

\newcommand{\alt}[1]{#1'}							
\newcommand{\dual}[1]{#1^{\ast}}						
\newcommand{\est}[1]{\hat #1}						
\newcommand{\sol}[1][\state]{#1^{\ast}}					

\DeclareMathOperator{\bigoh}{\mathcal O}						

\newcommand{\R}{\mathbb{R}}						

\newcommand{\vdim}{\debug d}							
\newcommand{\vecspace}{\mathcal{\debug V}}					
\newcommand{\dspace}{\vecspace^{\ast}}					

\DeclareMathOperator{\dist}{dist}						

\newcommand{\base}{\debug p}							

\newcommand{\ball}{\mathbb{B}}						
\newcommand{\sphere}{\mathbb{S}}					

\newcommand{\nhd}{\mathcal{\debug U}}							

\DeclareMathOperator{\tcone}{TC}						
\DeclareMathOperator{\dcone}{\tcone^{\ast}}				
\DeclareMathOperator{\pcone}{\mathrm{PC}}				

\newcommand{\feas}{\mathcal{\debug X}}						
\newcommand{\sols}{\sol[\feas]}
\newcommand{\obj}{\debug f}							
\newcommand{\sobj}{\debug F}							

\newcommand{\state}{\debug x}							
\newcommand{\states}{\feas}
\newcommand{\dstate}{\debug y}							
\newcommand{\dstates}{\mathcal{\debug Y}}

\newcommand{\payv}{\debug v}							

\DeclareMathOperator{\Eucl}{\debug \Pi}						
\DeclareMathOperator{\logit}{\debug \Lambda}					

\newcommand{\fench}{\debug F}							
\newcommand{\hreg}{\debug h}							
\newcommand{\mirror}{\debug Q}							
\newcommand{\strong}{\debug K}							

\newcommand{\act}{\debug X}							
\newcommand{\score}{\debug Y}							
\newcommand{\step}{\debug \gamma}						

\DeclareMathOperator{\ex}{\mathbb{E}}					
\DeclareMathOperator{\prob}{\mathbb{P}}				
\DeclareMathOperator{\simplex}{\Delta}					

\newcommand{\as}{\textpar{a.s.}\xspace}				
\newcommand{\filter}{\mathcal{\debug F}}						
\newcommand{\noise}{\debug U}							
\newcommand{\snoise}{\debug \xi}							
\newcommand{\noisedev}{\debug \sigma}						
\newcommand{\noisevar}{\noisedev^{2}}					

\newcommand{\sample}{\debug \omega}						
\newcommand{\samples}{\debug \Omega}						

\providecommand\given{}							

\DeclarePairedDelimiterXPP{\exof}[1]{\ex}{[}{]}{}{
\renewcommand\given{\nonscript\,\delimsize\vert\nonscript\,\mathopen{}} #1}

\DeclarePairedDelimiterXPP{\probof}[1]{\prob}{(}{)}{}{
\renewcommand\given{\nonscript\,\delimsize\vert\nonscript\,\mathopen{}} #1}

\DeclareMathOperator*{\argmax}{arg\,max}						
\DeclareMathOperator*{\argmin}{arg\,min}						

\DeclareMathOperator{\im}{im}						

\newcommand{\from}{\colon}							

\newcommand{\start}{\debug 1}							
\newcommand{\running}{\debug 1,2,\dotsc}						
\newcommand{\run}{\debug n}							
\newcommand{\runalt}{\debug k}							

\newcommand{\eps}{\varepsilon}						
\newcommand{\pd}{\partial}						

\usepackage{algorithm}
\usepackage{algpseudocode}								

\theoremstyle{plain}
\newtheorem{theorem}{Theorem}
\newtheorem{corollary}[theorem]{Corollary}
\newtheorem*{corollary*}{Corollary}
\newtheorem{lemma}[theorem]{Lemma}
\newtheorem{proposition}[theorem]{Proposition}

\theoremstyle{definition}
\newtheorem{definition}[theorem]{Definition}
\newtheorem*{definition*}{Definition}
\newtheorem{assumption}{Assumption}

\theoremstyle{remark}

\newtheorem*{remark*}{Remark}
\newtheorem*{notation*}{Notational remark}
\newtheorem{example}{Example}

\newenvironment{proofof}[1]{\begin{proof}[#1]}{\end{proof}}

\numberwithin{equation}{section}
\numberwithin{theorem}{section}
\numberwithin{remark}{section}
\numberwithin{example}{section}

\newcounter{proofstep}
\newenvironment{proofstep}[1]
{\vspace{3pt}
\refstepcounter{proofstep}%
\par\textit{Step~\arabic{proofstep}:~#1}.\,}
{\vspace{2pt}}

\usepackage[textwidth=30mm]{todonotes}					

\newcommand{\debug}[1]{#1}						

\newcommand{\const}{\debug c}

\newcommand{\argdot}{\cdot}
\newcommand{\intsimplex}{\simplex^{\circ}}
\newcommand{\region}{\mathcal{\debug C}}
\newcommand{\pspace}{(\samples,\filter,\prob)}

\newcommand{\fball}{\ball_{\fench}}
\newcommand{\vbound}{\debug V}

\newcommand{\dsol}{\dstate^{\ast}}

\begin{document}


\title
[Mirror descent beyond convex programming]
{On the convergence of mirror descent\\beyond stochastic convex programming}

\author
[Z.~Zhou]
{Zhengyuan Zhou$^{\ast}$}
\address{$^{\ast}$\,%
Department of Electrical Engineering, Stanford University.}
\author
[P.~Mertikopoulos]
{Panayotis Mertikopoulos$^{\S}$}
\address{$^{\S}$\,%
Univ. Grenoble Alpes, CNRS, Inria, LIG, 38000, Grenoble, France.}
\author
[N.~Bambos]
{\\Nicholas Bambos$^{\ast}$}
\author
[S.~Boyd]
{Stephen Boyd$^{\ast}$}
\author
[P.~W.~Glynn]
{Peter W.~Glynn$^{\ast}$}

%
%
\thanks{%
Part of this work was presented in the 31st International Conference on Neural Information Processing Systems (NIPS 2017) \cite{ZMBB+17-NIPS}.
P.~Mertikopoulos was partially supported by
the French National Research Agency (ANR) grant
ORACLESS (ANR\textendash 16\textendash CE33\textendash 0004\textendash 01).}
\thanks{The authors are indebted to two anonymous referees for their insightful suggestions and remarks.}


\newcommand{\acli}[1]{\textit{\acl{#1}}}
\newcommand{\acdef}[1]{\textit{\acl{#1}} \textup{(\acs{#1})}\acused{#1}}
\newcommand{\acdefp}[1]{\emph{\aclp{#1}} \textup(\acsp{#1}\textup)\acused{#1}}

\newacro{OD}[O/D]{origin-destination}
\newacro{VI}{variational inequality}
\newacroplural{VI}{variational inequalities}
\newacro{GAN}{generative adversarial network}
\newacro{EW}{exponential weights}
\newacro{VC}{variational coherence}
\newacro{SMD}{stochastic mirror descent}
\newacro{SGD}{stochastic gradient descent}
\newacro{ODE}{ordinary differential equation}
\newacro{APT}{asymptotic pseudotrajectory}
\newacroplural{APT}{asymptotic pseudotrajectories}
\newacro{MD}{mirror descent}
\newacro{DA}{dual averaging}
\newacro{LHS}{left-hand side}
\newacro{RHS}{right-hand side}
\newacro{NE}{Nash equilibrium}
\newacroplural{NE}[NE]{Nash equilibria}
\newacro{SA}{stochastic approximation}
\newacro{VS}{variational stability}
\newacro{LLN}{law of large numbers}
\newacro{MDS}{martingale difference sequence}
\newacro{iid}[i.i.d.]{independent and identically distributed}


\begin{abstract}
\acused{SMD}
In this paper, we examine the convergence of mirror descent in a class of stochastic optimization problems that are not necessarily convex (or even quasi-convex), and which we call \emph{variationally coherent}.
Since the standard technique of ``ergodic averaging'' offers no tangible benefits beyond convex programming, we focus directly on the algorithm's last generated sample (its ``last iterate''), and we show that it converges with probabiility $1$ if the underlying problem is coherent.
We further consider a localized version of \acl{VC} which ensures local convergence of \ac{SMD} with high probability.
These results contribute to the landscape of nonconvex stochastic optimization by showing that \mbox{(quasi-)}convexity is not essential for convergence to a global minimum:
rather, \acl{VC}, a much weaker requirement, suffices.
Finally, building on the above, we reveal an interesting insight regarding the convergence speed of \ac{SMD}:
in problems with sharp minima (such as generic linear programs or concave minimization problems), \ac{SMD} reaches a minimum point \emph{in a finite number of steps} \as, even in the presence of persistent gradient noise.
This result is to be contrasted with existing black-box convergence rate estimates that are only asymptotic.
\end{abstract}

\maketitle
\renewcommand{\sharp}{\debug\rho}
\acresetall
\acused{iid}
\allowdisplaybreaks

\section{Introduction}
\label{sec:introduction}

Stochastic mirror descent (\acs{SMD})\acused{SMD}
and its variants arguably comprise one of the most widely used families of first-order methods in stochastic optimization \textendash\ convex and non-convex alike \cite{NY83,BecTeb03,CT93,CLP12,DAJJ12,LNS12,MBNS17,MS18,MZ18,NL14,Nem04,NJLS09,Nes09,SS11,Xia10}.
Heuristically, in the ``\acl{DA}'' (or ``lazy'') incarnation of the method \cite{Nes09,Xia10,SS11},
\ac{SMD} proceeds by aggregating a sequence of \ac{iid} gradient samples and then mapping the result back to the problem's feasible region via a specially constructed ``mirror map'' (the namesake of the method).
In so doing, \ac{SMD} generalizes and extends the classical \ac{SGD} algorithm (with Euclidean projections playing the role of the mirror map) \cite{RM51,Pol87,Nes04},
the exponentiated gradient method of \cite{KW97},
the matrix regularization schemes of \cite{TRW05,KSST12,MBNS17},
and many others.

Starting with the seminal work of Nemirovski and Yudin \cite{NY83}, the convergence of \acl{MD} has been studied extensively in the context of
convex programming (including distributed and stochastic optimization problems) \cite{BecTeb03,NJLS09,Nes09,Xia10},
non-cooperative games\hspace{1pt}/\hspace{1pt}saddle-point problems \cite{NJLS09,Nes09,MZ18},
and
monotone \acp{VI} \cite{Nem04,Nes09,JNT11}.
In this monotone setting, it is customary to consider the so-called ``ergodic average'' $\bar\act_{\run} = \sum_{\runalt=1}^{\run} \step_{\runalt} \act_{\runalt} \big/ \sum_{\runalt=1}^{\run} \step_{\runalt}$ of the algorithm's generated sample points $\act_{\run}$, with $\step_{\run}$ denoting the method's step-size.
The reason for this is that, by Jensen's inequality, convexity guarantees that a regret-based analysis can lead to explicit convergence rates for $\bar\act_{\run}$ \cite{NJLS09,Nes09,Xia10,SS11}.
However, this type of averaging provides no tangible benefits in non-convex programs, so it seems more natural to focus directly on the algorithm's last generated sample \textendash\ its ``last iterate''.



The long-term behavior of the last iterate of \ac{SMD} was recently studied by Shamir and Zhang \cite{SZ13} and Nedic and Lee \cite{NL14} in the context of strongly convex problems.
In this case, the algorithm's last iterate achieves the same value convergence rate as its ergodic average, so averaging is not more advantageous.
Jiang and Xu \cite{JX08} also examined the convergence of the last iterate of \ac{SGD} in a class of (not necessarily monotone) \aclp{VI} that admit a unique solution, and they showed that it converges to said solution with probability $1$.
In a very recent paper \cite{chen2018}, it was shown that in phase retrieval problems (a special class of non-convex problems that involve systems of quadratic equations), \ac{SGD} with random initialization converges to global optimal solutions with probability $1$.
Finally, in general non-convex problems, Ghadimi and Lan \cite{GL13,GL16} showed that running \ac{SGD} with a randomized stopping time guarantees convergence to a critical point in the mean, and they estimated the speed of this convergence.
However, beyond these (mostly recent) results, not much is known about the convergence of the individual iterates of \acl{MD} in non-convex programs.

\subsection*{Our contributions}
\label{sec:contributions}

In this paper, we examine the asymptotic behavior of \acl{MD} in a class of stochastic optimization problems that are not necessarily convex (or even quasi-convex).
This class of problems, which we call \emph{variationally coherent} (\acs{VC}), are related to a class of \aclp{VI} studied by Jiang and Xu \cite{JX08} and, earlier, by Wang et al. \cite{WXW01} \textendash\ though, importantly, we do not assume here the existence of a unique solution.
Focusing for concreteness on the \acdef{DA} variant of \ac{SMD} (also known as ``lazy'' \acl{MD}) \cite{Nes09,Xia10,SS11}, we show that the algorithm's last iterate converges to a global minimum with probability $1$ under mild assumptions for the algorithm's gradient oracle (unbiased \ac{iid} gradient samples that are bounded in $L^{2}$).
This result can be seen as the ``mirror image'' of the analysis of \cite{JX08} and reaffirms that \mbox{(quasi-)}convexity/monotonicity is not essential for convergence to a global optimum point:
the weaker requirement of \acl{VC} suffices.

To extend the range of our analysis, we also consider a localized version of \acl{VC} which includes multi-modal functions that  are not even \emph{locally} \mbox{(quasi-)}convex near their minimum points (so, in particular, an eigenvalue-based analysis cannot be readily applied to such problems).
Here, in contrast to the globally coherent case, a single, ``unlucky'' gradient sample could drive the algorithm away from the ``basin of attraction'' of a local minimum (even a locally coherent one), possibly never to return.
Nevertheless, we show that, with overwhelming probability, the last iterate of \ac{SMD} converges locally to minimum points that are locally coherent (for a precise statement, see \cref{sec:local}).

Going beyond this ``black-box'' analysis, we also consider a class of optimization problems that admit \emph{sharp} minima, a fundamental notion due to Polyak \cite{Pol87}.
In stark contrast to existing ergodic convergence rates (which are asymptotic in nature),
we show that the last iterate of \ac{SMD} converges to sharp minima of variationally coherent problems in an almost surely \emph{finite} number of iterations, provided that the method's mirror map is surjective.
As an important corollary, it follows that the last iterate of (lazy) \acl{SGD} attains a solution of a stochastic linear program in a finite number of steps \as.
For completeness, we also derive a localized version of this result for problems with sharp local minima that are not globally coherent:
in this case, convergence in a finite number of steps is retained but, instead of ``almost surely'', convergence now occurs with overwhelming probability.

Important classes of problems that admit sharp minima are generic linear programs (for the global case) and concave minimization problems (for the local case).
In both instances, the (fairly surprising) fact that \ac{SMD} attains a minimizer in a finite number of iterations should be contrasted to existing work on stochastic linear programming which exhibits asymptotic convergence rates \cite{ARH97,Van07}.
We find this result particularly appealing as it highlights an important benefit of working with ``lazy'' descent schemes:
``greedy'' methods (such as vanilla gradient descent) always take a gradient step from the last generated sample, so convergence in a finite number of iterations is a priori impossible in the presence of persistent noise.
By contrast, the aggregation of gradient steps in ``lazy'' schemes means that even a ``bad'' gradient sample might not change the algorithm's sampling point (if the mirror map is surjective), so finite-time convergence \emph{is} possible in this case.

Our analysis hinges on the construction of a primal-dual analogue of the Bregman divergence which we call the \emph{Fenchel coupling}, and which tracks the evolution of the algorithm's (dual) gradient aggregation variable relative to a target point in the problem's (primal) feasible region.
This energy function allows us to perform a quasi-Fejérian analysis of \acl{SMD}
and, combined with a series of (sub)martingale convergence arguments, ultimately yields the convergence of the algorithm's last iterate \textendash\ first as a subsequence, then with probability $1$.

\section{Problem setup and basic definitions}
\label{sec:setup}

\subsection{The main problem}
\label{subsec:problem}

Let $\states$ be a convex compact subset of a $\vdim$-di\-men\-sio\-nal vector space $\vecspace$ with norm $\norm{\argdot}$.
Throughout this paper, we will focus on stochastic optimization problems of the general form
\begin{equation}
\label{eq:opt}
\tag{Opt}
\begin{aligned}
\text{minimize}
&\quad
\obj(\state),
\\
\text{subject to}
&\quad
\state\in\states,
\end{aligned}
\end{equation}
where
\begin{equation}
\label{eq:obj}
\obj(\state)
	= \exof{\sobj(\state;\sample)}
\end{equation}
for some stochastic objective function $\sobj\from\states\times\samples\to\R$ defined on an underlying (complete) probability space $\pspace$.
In terms of regularity, our blanket assumptions for \eqref{eq:opt} will be as follows:

\begin{assumption}
\label{asm:diff}
$\sobj(\state, \sample)$ is continuously differentiable in $\state$ for almost all $\sample\in\samples$.
\end{assumption}

\begin{assumption}
\label{asm:L2}
The gradient of $\sobj$ is uniformly bounded in $L^{2}$, \ie $\exof{\dnorm{\nabla\sobj(\state;\sample)}^{2}} \leq \vbound^{2}$ for some finite $\vbound\geq0$ and all $\state\in\states$.
\end{assumption}
\vspace{1ex}

\begin{remark*}
In the above, gradients are treated as elements of the dual space $\dstates\equiv\dual\vecspace$ of $\vecspace$, and $\dnorm{\dstate} = \sup\setdef{\braket{\dstate}{\state}}{\norm{\state}\leq1}$ denotes the dual norm of $\dstate\in\dstates$.
We also note that $\nabla\sobj(\state;\sample)$ refers to the gradient of $\sobj(\state;\sample)$ with respect to $\state$;
since $\samples$ is not assumed to carry a differential structure, there is no danger of confusion.
\end{remark*}
\vspace{1ex}

\cref{asm:diff} is a token regularity assumption which can be relaxed to account for nonsmooth objectives by using subgradient devices (as opposed to gradients).
However, this would make the presentation significantly more cumbersome, so we stick with smooth objectives throughout.
\cref{asm:L2} is also standard in the stochastic optimization literature:
it holds trivially if $\sobj$ is uniformly Lipschitz (another commonly used condition) and, by the dominated convergence theorem, it further implies that $\obj$ is smooth and $\nabla\obj(\state) = \nabla\exof{\sobj(\state;\sample)} = \exof{\nabla\sobj(\state;\sample)}$ is bounded.
As a result, the solution set
\begin{equation}
\label{eq:solset}
\sols
	= \argmin\obj
\end{equation}
of \eqref{eq:opt} is closed and nonempty (by the compactness of $\states$ and the continuity of $\obj$).

\vspace{1ex}
We briefly discuss below two important examples of \eqref{eq:opt}:
\vspace{1ex}

\begin{example}[Distributed optimization]
\label{ex:distributed}
An important special case of \eqref{eq:opt} with high relevance to statistical inference, signal processing and machine learning is when $\obj$ is of the special form
\begin{equation}
\label{eq:obj-distributed}
\obj(\state)
	= \frac{1}{N} \sum_{i=1}^{N} \obj_{i}(\state),
\end{equation}
for some family of functions (or training samples) $\obj_{i}\from\feas\to\R$, $i=1,\dotsc,N$.
As an example, this setup corresponds to empirical risk minimization with uniform weights, the sample index $i$ being drawn with uniform probability from $\{1,\dotsc,N\}$.
\end{example}
\vspace{1ex}

\begin{example}[Noisy gradient measurements]
\label{ex:noisy}
Another widely studied instance of \eqref{eq:opt} is when
\begin{equation}
\label{eq:sobj-noisy}
\sobj(\state;\noise)
	= \obj(\state) + \braket{\noise}{x}
\end{equation}
for some random vector $\noise$ such that $\exof{\noise} = 0$ and $\exof{\dnorm{\noise}^{2}} < \infty$.
This gives $\nabla\sobj(\state;\noise) = \nabla\obj(\state) + \noise$, so \eqref{eq:opt} can be seen here as a model for deterministic optimization problems with noisy gradient measurements.
\end{example}
\vspace{1ex}

\subsection{Variational coherence}
\label{subsec:VC}

We are now in a position to define the class of variationally coherent problems:

\begin{definition}
\label{def:VC}
We say that \eqref{eq:opt} is \emph{variationally coherent} if
\begin{equation}
\label{eq:VC}
\tag{VC}
\braket{\nabla\obj(\state)}{\state - \sol}
	\geq 0
	\quad
	\text{for all $\state\in\states$, $\sol\in\sols$},
\end{equation}
and there exists some $\sol\in\sols$ such that equality holds in \eqref{eq:VC} only if $\state \in \sols$.
\end{definition} 

In words, \eqref{eq:VC} states that solutions of \eqref{eq:opt} can be harvested by solving a (Minty) \acl{VI} \textendash\ hence the term ``variational coherence''.
To the best of our knowledge, the closest analogue to this condition first appeared in the classical paper of Bottou \cite{Bot98} on online learning and stochastic approximation algorithms, but with the added assumptions that
\begin{inparaenum}
[\itshape a\upshape)]
\item
the problem \eqref{eq:opt} admits a unique solution $\sol$;
and
\item
an extra positivity requirement for $\braket{\nabla\obj(\state)}{\state - \sol}$ in punctured neighborhoods of $\sol$.
\end{inparaenum}
In the context of \aclp{VI}, a closely related variant of \eqref{eq:VC} has been used to establish the convergence of extra-gradient methods \cite{WXW01,FP03} and \acl{SGD} \cite{JX08} in (Stampacchia) \aclp{VI} with a unique solution.
By contrast, there is no uniqueness requirement in \eqref{eq:VC}, an aspect of the definition which we examine in more detail below.

We should also note that, as stated, \eqref{eq:VC} is a non-random requirement for $\obj$ so it applies equally well to \emph{deterministic} optimization problems.
Alternatively, by the dominated convergence theorem, \eqref{eq:VC} can be written equivalently as
\begin{equation}
\label{eq:VC-stoch}
\exof{\braket{\nabla\sobj(\state;\sample)}{\state - \sol}}
	\geq 0,
\end{equation}
so it can be interpreted as saying that $\sobj$ is variationally coherent ``on average'', without any individual realization thereof satisfying \eqref{eq:VC}.
Both interpretations will come in handy later on.

All in all, the notion of \acl{VC} will play a central role in our paper so a few examples are in order:
\vspace{1ex}

\begin{example}[Convex programming]
\label{ex:cvx}
If $\obj$ is convex, $\nabla\obj$ is \emph{monotone} \cite{RW98} in the sense that
\begin{equation}
\label{eq:mono}
\braket{\nabla\obj(\state) - \nabla\obj(\alt\state)}{\state - \alt\state} \geq 0
	\quad
	\text{for all $\state,\alt\state \in \states$}.
\end{equation}
By the first-order optimality conditions for $\obj$, it follows that $\braket{\obj(\sol)}{\state - \sol} \geq 0$ for all $\state\in\states$.
Hence, by monotonicity, we get
\begin{equation}
\label{eq:mono1}
\braket{\nabla\obj(\state)}{\state - \sol}
	\geq \braket{\nabla\obj(\sol)}{\state - \sol}
	\geq 0
	\quad
	\text{for all $\state\in\states$, $\sol\in\sols$}.
\end{equation}
By convexity, it further follows that $\braket{\nabla\obj(\state)}{\state - \sol} < 0$ whenever $\sol\in\sols$ and $\state\in\states\setminus\sols$, so equality holds in \eqref{eq:mono1} if and only if $\state \in \sols$.
This shows that convex programs automatically satisfy \eqref{eq:VC}.
\end{example}
\vspace{1ex}

\begin{example}[Quasi-convex problems]
\label{ex:quasi}
More generally, the above analysis also extends to \emph{quasi-convex} objectives, \ie when
\begin{equation}
\label{eq:quasi}
\tag{QC}
\obj(\alt\state)
	\leq \obj(\state)
	\implies
\braket{\nabla\obj(\state)}{\alt\state-\state}
	\leq 0
\end{equation}
for all $\state,\alt\state\in\states$ \cite{BV04}.
In this case, we have:
\end{example}

\begin{proposition}
\label{prop:quasi}
Suppose that $\obj$ is quasi-convex and non-degenerate, \ie
\begin{equation}
\label{eq:non-deg}
\braket{\nabla\obj(\state)}{z}
	\neq 0
	\quad
	\text{for all nonzero $z\in\tcone(\state)$, $\state\in\states\setminus\sols$}.
\end{equation}
Then, $\obj$ is variationally coherent.
\end{proposition}

\begin{remark*}
The non-degeneracy condition \eqref{eq:non-deg} is \emph{generic} in that it is satisfied by every quasi-convex function after an arbitrarily small perturbation leaving its minimum set unchanged.
In particular, it is automatically satisfied if $\obj$ is convex or pseudo-convex.
\end{remark*}

\begin{proof}
Take some $\sol\in\sols$ and $\state\in\states$.
Then, letting $\alt\state=\sol$ in \eqref{eq:quasi}, we readily obtain $\braket{\nabla\obj(\state)}{\state - \sol} \geq 0$ for all $\state\in\states$, $\sol\in\sols$.
Furthermore, if $\state\notin\sols$ but $\braket{\nabla\obj(\state)}{\state - \sol} = 0$, the gradient non-degeneracy condition \eqref{eq:non-deg} would be violated, implying in turn that, for any $\sol\in\sols$, we have $\braket{\nabla\obj(\state)}{\state - \sol} = 0$ only if $\state\in\sols$.
This shows that $\obj$ satisfies \eqref{eq:VC}.
\end{proof}
\vspace{1ex}

\begin{example}[Beyond quasi-convexity]
\label{ex:non-convex}
A simple example of a function that is variationally coherent without even being quasi-convex is
\begin{equation}
\obj(\state)
	= 2 \sum_{i=1}^{\vdim} \sqrt{1 + \state_{i}},
	\quad
	\state\in[0,1]^{\vdim}.
\end{equation}
When $\vdim\geq2$, it is easy to see $\obj$ is not quasi-convex:
for instance, taking $\vdim = 2$, $\state = (0,1)$ and $\alt\state = (1, 0)$ yields $\obj(\state/2 + \alt\state/2) = 2 \sqrt{6} > 2 \sqrt{2} = \max\{\obj(\state),\obj(\alt\state)\}$, so $\obj$ is not quasi-convex.
On the other hand, to estabilish \eqref{eq:VC}, simply note that $\sols = \{0\}$ and $\braket{\nabla\obj(\state)}{\state - 0} = \sum_{i=1}^{\vdim} \state_{i}/\sqrt{1 + \state_{i}} > 0$ for all $\state \in [0,1]^{\vdim}\exclude{0}$. 
\end{example}
\vspace{1ex}

\begin{example}[A weaker version of coherence]
\label{ex:star}
Consider the function
\begin{equation}
\obj(\state)
	= \frac{1}{2} \prod_{i=1}^{\vdim} \state_{i}^{2},
	\quad
	\state\in[-1,1]^{\vdim}.
\end{equation}
By inspection, it is easy to see that the minimum set of $\obj$ is $\sols = \setdef{\sol\in[-1,1]^{\vdim}}{\sol_{i} = 0\;\text{for some $i=1,\dotsc,\vdim$}}$.%
\footnote{Linear combinations of functions of this type play an important role in training deep learning models \textendash\ and, in particular \acp{GAN} \cite{GPM+14}.}
Since $\sols$ is not convex for $\vdim\geq2$, $\obj$ is not quasi-convex.
On the other hand, we have $\nabla\obj(\state) =2 \obj(\state) \cdot (1/\state_{1},\dotsc,1/\state_{\vdim})$, so $\braket{\nabla\obj(\state)}{\state - 0} \geq 0$ for all $\state\in[-1,1]^{\vdim}$ with equality only if $\state\in\sols$.
Moreover, for any $\sol\in\sols$ and all $\state\in\states$ sufficiently close to $\sol$, we have
\begin{equation}
\braket{\nabla\obj(\state)}{\state - \sol}
	= 2\obj(\state) \sum_{i=1}^{\vdim} \bracks*{1 - \frac{\sol_{i}}{\state_{i}}}
	= 2\obj(\state) \bracks*{\vdim - \sum_{i:\sol_{i}\neq0} \frac{\sol_{i}}{\state_{i}}}
	\geq 0.
\end{equation}
We thus conclude that $\obj$ satisfies the following weaker version of \eqref{eq:VC}:
\end{example}

\begin{definition}
\label{def:VC-weak}
We say that $\obj\from\states\to\R$ is \emph{weakly coherent} if the following conditions are satisfied:
\begin{enumerate}
[\indent\itshape a\upshape)]
\item
There exists some $\base\in\sols$ such that $\braket{\nabla\obj(\state)}{\state - \base} \geq 0$ with equality only if $\state\in\sols$.
\item
For all $\sol\in\sols$, $\braket{\nabla\obj(\state)}{\state - \sol} \geq 0$ whenever $\state$ is close enough to $\sol$.
\end{enumerate}
\end{definition}
Our analysis also applies to problems satisfying these less stringent requirements, in which case the minimum set $\sols = \argmin\obj$ of $\obj$ need not even be convex.%
\footnote{Obviously, \cref{def:VC,def:VC-weak} coincide if $\sols$ is a singleton.
This highlights the intricacies that arise in problems that do not admit a unique solution.}
For simplicity, we will first work with \cref{def:VC} and relegate these considerations to \cref{sec:local}.



\subsection{Stochastic mirror descent}
\label{subsec:MD}

To solve \eqref{eq:opt}, we will focus on the family of algorithms known as \acdef{SMD}, a class of first-order methods pioneered by Nemirovski and Yudin \cite{NY83} and studied further by Beck and Teboulle \cite{BecTeb03}, Nesterov \cite{Nes09}, Lan et al. \cite{LNS12}, and many others.
Referring to \cite{SS11,Bub15} for an overview, the specific variant of \ac{SMD} that we consider here is usually referred to as \acdef{DA} \cite{Nes09,Xia10,MZ18} or ``\emph{lazy}'' \acl{MD} \cite{SS11}.

The main idea of the method is as follows:
At each iteration, the algorithm takes as input an \ac{iid} sample of the gradient of $\sobj$ at the algorithm's current state.
Subsequently, the method takes a step along this stochastic gradient in the dual space $\dstates\equiv\dspace$ of $\vecspace$ (where gradients live), the result is ``mirrored'' back to the problem's feasible region $\states$, and the process repeats.
Formally, this gives rise to the recursion
\begin{equation}
\label{eq:SMD}
\tag{SMD}
\begin{aligned}
\act_{\run}
	&= \mirror(\score_{\run})
	\\
\score_{\run+1}
	&= \score_{\run} - \step_{\run} \nabla\sobj(\act_{\run};\sample_{\run})
\end{aligned}
\end{equation}
where:
\begin{enumerate}
\item
$\run=\running$ denotes the algorithm's running counter.
\item
$\score_{\run}\in\dstates$ is a score variable that aggregates gradient steps up to stage $\run$.
\item
$\mirror\from\dstates\to\states$ is the \emph{mirror map} that outputs a solution candidate $\act_{\run}\in\states$ as a function of the score variable $\score_{\run}\in\dspace$.
\item
$\sample_{\run}\in\samples$ is a sequence of \ac{iid} samples.%
\footnote{The indexing convention for $\sample_{\run}$ means that $\score_{\run}$ and $\act_{\run}$ are \emph{predictable} relative to the natural filtration $\filter_{\run} = \sigma(\sample_{1},\dotsc,\sample_{\run})$ of $\sample_{\run}$, \ie $\score_{\run+1}$ and $\act_{\run+1}$ are both $\filter_{\run}$-measurable.
To this history, we also attach the trivial $\sigma$-algebra as $\filter_{0}$ for completeness.}
\item
$\step_{\run}>0$ is the algorithm's step-size sequence, assumed in what follows to satisfy the Robbins\textendash Monro summability condition
\begin{equation}
\label{eq:L2L1}
\sum_{\run=1}^{\infty} \step_{\run}^{2}
	<\infty
	\quad
	\text{and}
	\quad
	\sum_{\run=1}^{\infty} \step_{\run}
	= \infty.
\end{equation}
\end{enumerate}
For a schematic illustration and a pseudocode implementation of \eqref{eq:SMD}, see \cref{fig:SMD} and \cref{alg:SMD} respectively.
\medskip


\begin{algorithm}[tbp]
\caption{Stochastic mirror descent (\ac{SMD})}
\label{alg:SMD}

\tt
\begin{algorithmic}[1]
\Require
	mirror map $\mirror\from\dstates\to\states$;
	step-size sequence $\step_{\run}>0$
\State
	choose $\score\in\dstates \equiv \dspace$
	\Comment{initialization}
\For{$\run=\running$}
	\State
		set $\act \leftarrow \mirror(\score)$
		\Comment{set state}
	\State
		draw $\sample\in\samples$
		\Comment{gradient sample}
	\State
		get $\est\payv = -\nabla\sobj(\act;\sample)$
		\Comment{get oracle feedback}
	\State
		set $\score \leftarrow \score + \step_{\run} \est\payv$
		\Comment{update score variable}
\EndFor
\\
\Return $\act$
	\Comment{output}
\end{algorithmic}
\end{algorithm}



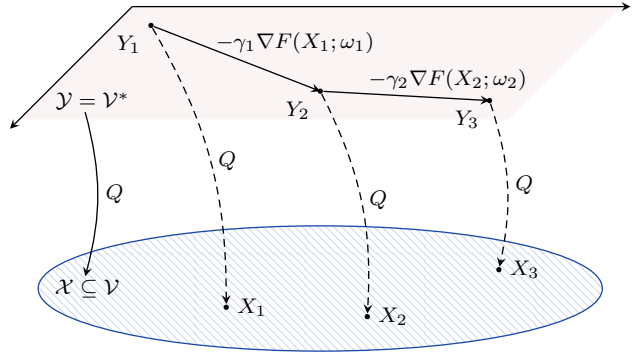
\begin{figure}[tbp]
\centering
\footnotesize

\colorlet{TangentColor}{DodgerBlue!40!MidnightBlue}
\colorlet{PolarColor}{FireBrick}

\begin{tikzpicture}
[auto,
>=stealth,
vecstyle/.style = {->, line width=.5pt},
edgestyle/.style={-, line width=.5pt},
nodestyle/.style = {circle,fill=black,inner sep=0, minimum size=2},
plotstyle/.style={color=DarkGreen!80!Cyan,thick}]

\def\radius{2.5}
\def\costhirty{0.8660256}
\def\cosfortyfive{0.7071068}
\def\diff{0.77886966103678}
\def\veclength{.5}
\def\conescale{.8}

\coordinate (base) at (0,0);
\coordinate (dbase) at (0,1.5*\radius);

\coordinate (X) at ($(base) + (-.25*\radius,-0*\radius)$);
\coordinate (Y) at ($(dbase) + (-.25*\radius,-.5*\radius)$);

\coordinate (y1) at ($(dbase) + (0.1*\radius,-.1*\radius)$);
\coordinate (y2) at ($(dbase) + (1*\radius,-.45*\radius)$);
\coordinate (y3) at ($(dbase) + (1.9*\radius,-.5*\radius)$);
\coordinate (y4) at ($(dbase) + (2.25*\radius,-.1*\radius)$);

\coordinate (x1) at ($(base) + (.5*\radius,-.1*\radius)$);
\coordinate (x2) at ($(base) + (1.25*\radius,-.15*\radius)$);
\coordinate (x3) at ($(base) + (1.95*\radius,.1*\radius)$);
\coordinate (x4) at ($(base) + (2.25*\radius,0.1*\radius)$);

\filldraw
[pattern = north west lines, pattern color = TangentColor!20, draw=TangentColor, edgestyle]
($(base) + (\radius,0)$) ellipse [x radius= 1.5*\radius, y radius= \radius/3];

\fill
[PolarColor!5]
(dbase) -- ($(dbase) - (.6*\radius,.6*\radius)$) -- ($(dbase) + (2*\radius,-.6*\radius)$)-- ($(dbase) + (2.6*\radius,0)$);
\draw
[vecstyle]
(dbase) -- ($(dbase) + (2.65*\radius,0)$);
\draw
[vecstyle]
(dbase) -- ($(dbase) - .65*(\radius,\radius)$);

\node [inner sep = 1pt] (X) at (X) {$\feas \subseteq \vecspace\!$};
\node [inner sep = 1pt] (Y) at (Y) {$\dstates = \dspace\!\!\!\!$};
\draw [vecstyle] (Y.south) to [bend left =15] node {$\mirror$}(X.north);

\node [nodestyle, label = below left:$\score_{1}$] (y1) at (y1) {};
\node [nodestyle, label = below left:$\score_{2}$] (y2) at (y2) {};
\node [nodestyle, label = below left:$\score_{3}$] (y3) at (y3) {};

\draw [vecstyle] (y1) -- (y2.center) node [right,near start,black] {\;\;$-\step_{1} \nabla\sobj(\act_{1};\sample_{1})$};
\draw [vecstyle] (y2) -- (y3.center) node [above,near end,black] {$-\step_{2} \nabla\sobj(\act_{2};\sample_{2})$};

\node [nodestyle, label = right:$\act_{1}$] (x1) at (x1) {};
\node [nodestyle, label = right:$\act_{2}$] (x2) at (x2) {};
\node [nodestyle, label = right:$\act_{3}$] (x3) at (x3) {};

\draw [vecstyle, densely dashed] (y1) to [bend left = 15] node [right] {$\mirror$} (x1);
\draw [vecstyle, densely dashed] (y2) to [bend left = 15] node [right] {$\mirror$} (x2);
\draw [vecstyle, densely dashed] (y3) to [bend left = 15] node [right] {$\mirror$} (x3);


\end{tikzpicture}
\caption{Schematic representation of \ac{SMD} (\cref{alg:SMD}).}
\label{fig:SMD}
\end{figure}


In more detail, the algorithm's mirror map $\mirror\from\dstates\to\states$ is defined as
\begin{equation}
\label{eq:mirror}
\mirror(\dstate)
	= \argmax_{\state\in\states} \{\braket{\dstate}{\state} - \hreg(\state)\},
\end{equation}
where the \emph{regularizer} (or \emph{penalty function}) $\hreg\from\states\to\R$ is assumed to be continuous and strongly convex on $\states$, \ie there exists some $\strong>0$ such that
\begin{equation}
\label{eq:strong}
\hreg(\tau\state + (1-\tau)\alt\state)
	\leq \tau \hreg(\state) + (1-\tau) \hreg(\alt\state) - \tfrac{1}{2} \strong\tau(1-\tau) \norm{\alt\state - \state}^{2},
\end{equation}
for all $\state,\alt\state\in\states$ and all $\tau\in[0,1]$.
The mapping $\mirror\from\dspace\to\states$ defined by \eqref{eq:mirror} is then called the \emph{mirror map induced by $\hreg$}.
For concreteness, we present below some well-known examples of regularizers and mirror maps:
\vspace{1ex}

\begin{example}[Euclidean regularization]
\label{ex:Eucl}
Let $\hreg(\state) = \frac{1}{2} \norm{\state}_{2}^{2}$.
Then, $\hreg$ is $1$-strongly convex with respect to the Euclidean norm $\norm{\cdot}_{2}$, and the induced mirror map is the closest point projection
\begin{equation}
\label{eq:mirror-Eucl}
\Eucl(\dstate)
	= \argmax_{\state\in\states} \braces[\big]{ \braket{\dstate}{\state} - \tfrac{1}{2} \norm{\state}_{2}^{2} }
	= \argmin_{\state\in\states}\; \norm{\dstate - \state}_{2}^{2}.
\end{equation}
The resulting descent algorithm is known in the literature as (lazy) \acdef{SGD} and we study it in detail in \cref{sec:sharp}.
For future reference, we also note that $\hreg$ is differentiable throughout $\states$ and $\Eucl$ is \emph{surjective} (\ie $\im\Eucl = \states$).
\end{example}
\vspace{1ex}

\begin{example}[Entropic regularization]
\label{ex:logit}
Let $\simplex = \setdef{\state\in\R^{\vdim}_{+}}{\sum_{i=1}^{\vdim} \state_{i} = 1}$ denote the unit simplex of $\R^{\vdim}$.
A widely used regularizer in this setting is the (negative) Gibbs entropy $\hreg(\state) = \sum_{i=1}^{\vdim} \state_{i} \log \state_{i}$:
this regularizer is $1$-strongly convex with respect to the $L^{1}$-norm and a straightforward calculation shows that the induced mirror map is
\begin{equation}
\label{eq:mirror-logit}
\logit(\dstate)
	= \frac{1}{\sum_{i=1}^{\vdim} \exp(\dstate_{i})} (\exp(\dstate_{1}),\dotsc,\exp(\dstate_{\vdim})).
\end{equation}
This example is known as \emph{entropic regularization} and the resulting \acl{MD} algorithm has been studied extensively in the context of linear programming, online learning and game theory \cite{SS11,AHK12}.
For posterity, we also note that $\hreg$ is differentiable \emph{only} on the relative interior $\intsimplex$ of $\simplex$ and $\im\logit = \intsimplex$ (\ie $\logit$ is ``essentially'' surjective).
\end{example}


\subsection{Overview of main results}
\label{subsec:main}

To motivate the analysis to follow, we provide below a brief overview of our main results:

\begin{itemize}

\item
\emph{Global convergence:}
If \eqref{eq:opt} is variationally coherent, the last iterate $\act_{\run}$ of \eqref{eq:SMD} converges to a global minimizer of $\obj$ with probability $1$.

\item
\emph{Local convergence:}
If $\sol$ is a \emph{locally coherent} minimum point of $\obj$ (a notion introduced in \cref{sec:local}), the last iterate $\act_{\run}$ of \ac{SMD} converges locally to $\sol$ with high probability.

\item
\emph{Sharp minima:}
If $\mirror$ is surjective and $\sol$ is a \emph{sharp} minimum of $\obj$ (a fundamental notion due to Polyak which we discuss in \cref{sec:sharp}), $\act_{\run}$ reaches $\sol$ in a \emph{finite} number of iterations \as.
\end{itemize}

\section{Main tools and first results}
\label{sec:recurrence}

As a stepping stone to analyze the long-term behavior of \eqref{eq:SMD}, we derive in this section a recurrence result which is interesting in its own right.
Specifically, we show that, if \eqref{eq:opt} is coherent, then, with probability $1$, $\act_{\run}$ visits any neighborhood of $\sols$ infinitely often;
as a corollary, there exists a (random) subsequence $\act_{\run_{\runalt}}$ of $\act_{\run}$ that converges to $\argmin\obj$ \as.
In what follows, our goal will be to state this result formally and to introduce the analytic machinery used for its proof (as well as the analysis of the subsequent sections).

\subsection{The Fenchel coupling}
\label{subsec:Fenchel}

The first key ingredient of our analysis will be the \emph{Fenchel coupling}, a primal-dual variant of the Bregman divergence \cite{Bre67} that plays the role of an energy function for \eqref{eq:SMD}:

\begin{definition}
\label{def:Fenchel}
Let $h\from\states \to \R$ be a regularizer on $\states$.
The induced \emph{Fenchel coupling} $\fench(\base,\dstate)$ between a base-point $\base\in\states$ and a dual vector $\dstate\in\dstates$ is defined as
\begin{equation}
\label{eq:Fenchel}
\fench(\base,\dstate)
	= \hreg(\base) + \hreg^{\ast}(\dstate) - \braket{\dstate}{\base}
\end{equation}
where $\hreg^{\ast}(\dstate) = \max_{\state\in\states} \{ \braket{\dstate}{\state} - \hreg(\state) \}$ denotes the convex conjugate of $h$.
\end{definition}

By Fenchel's inequality (the namesake of the Fenchel coupling), we have $\hreg(\base) + \hreg^{\ast}(\dstate) - \braket{\dstate}{\base} \geq 0$ with equality if and only if $\base = \mirror(\dstate)$.
As such, $\fench(\base,\dstate)$ can be seen as a (typically asymmetric) ``distance measure'' between $\base\in\states$ and $\dstate\in\dstates$.
The following lemma quantifies some basic properties of this coupling:

\begin{lemma}
\label{lem:Fenchel}
Let $h$ be a $\strong$-strongly convex regularizer on $\states$.
Then, for all $\base\in\states$ and all $\dstate,\alt\dstate\in\dstates$, we have:
\begin{subequations}
\label{eq:Fench-properties}
\begin{alignat}{2}
\label{eq:Fench-norm}
&a)
	\;\;
	\fench(\base,\dstate)
	&&\geq \frac{\strong}{2} \, \norm{\mirror(\dstate) - \base}^{2}.
	\hspace{17em}
	\\[3pt]
\label{eq:Fench-bound}
&b)
	\;\;
	\fench(\base,\alt\dstate)
	&&\leq \fench(\base,\dstate) + \braket{\alt\dstate - \dstate}{\mirror(\dstate) - \base} + \frac{1}{2\strong} \dnorm{\alt\dstate - \dstate}^{2}.
\end{alignat}
\end{subequations}
\end{lemma}

\cref{lem:Fenchel} (which we prove in \cref{app:proofs}) shows that $\mirror(\dstate_{\run})\to\base$ whenever $\fench(\base,\dstate_{\run})\to0$, so the Fenchel coupling can be used to test the convergence of the primal sequence $\state_{\run} = \mirror(\dstate_{\run})$ to a given base point $\base\in\states$.
For technical reasons, it will be convenient to also make the converse assumption, namely:

\begin{assumption}
\label{asm:reciprocity}
$\fench(\base,\dstate_{\run}) \to 0$ whenever $\mirror(\dstate_{\run})\to\base$.
\end{assumption}

\cref{asm:reciprocity} can be seen as a ``reciprocity condition'':
essentially, it means that the sublevel sets of $\fench(\base,\cdot)$ are mapped under $\mirror$ to neighborhoods of $\base$ in $\states$ (cf.~\cref{app:proofs}).
In this way, \cref{asm:reciprocity} can be seen as a primal-dual analogue of the reciprocity conditions for the Bregman divergence that are widely used in the literature on proximal and forward-backward methods \cite{CT93,Kiw97b}.
Most common regularizers satisfy this technical requirement
(including the Euclidean and entropic regularizers of \cref{ex:Eucl,ex:logit} respectively).

\subsection{Main recurrence result}
\label{subsec:recurrence}

To state our recurrence result, we require one last piece of notation pertaining to measuring distances in $\states$:

\begin{definition}
\label{def:metrics}
Let $\region$ be a subset of $\states$.
\begin{subequations}
\begin{enumerate}
\item
The \emph{distance} between $\region$ and $\state\in\states$ is defined as $\dist(\region,\state) = \inf_{\alt\state\in\region} \norm{\alt\state - \state}$,
and the corresponding \emph{$\eps$-neighborhood of $\region$} is
\begin{alignat}{2}
\label{eq:ball-norm}
\ball(\region,\eps)
	&= \setdef{\state\in\states}{\dist(\region,\state) < \eps}.
\intertext{%
\item
The (setwise) \emph{Fenchel coupling} between $\region$ and $\dstate\in\dstates$ is defined as $\fench(\region,\dstate) = \inf_{\state\in\region} \fench(\state,\dstate)$,
and the corresponding \emph{Fenchel $\delta$-zone} of $\region$ under $h$ is}
\label{eq:ball-Fench}
\fball(\region,\delta)
	&= \setdef{\state\in\states}{\text{$\state=\mirror(\dstate)$ for some $\dstate\in\dstates$ with $\fench(\region,\dstate) < \delta$}}.
\end{alignat}
\end{enumerate}
\end{subequations}
\end{definition} 

We then have the following recurrence result for variationally coherent problems:

\begin{proposition}
\label{prop:recurrence}
Fix some $\eps>0$ and $\delta>0$.
If \eqref{eq:opt} is variationally coherent and \cref{asm:diff,asm:L2,asm:reciprocity} hold, the \textpar{random} iterates $\act_{\run}$ of \cref{alg:SMD} enter $\ball(\sols,\eps)$ and $\fball(\sols,\delta)$ infinitely many times \as.
\end{proposition}

\begin{corollary}
\label{cor:recurrence}
With probability $1$, there exists a subsequence $\act_{\run_{\runalt}}$ of $\act_{\run}$ converging to a \textpar{random} minimum point $\sol$ of \eqref{eq:opt}.
\end{corollary}

\smallskip

The proof of \cref{prop:recurrence} consists of three main steps which we outline below:

\setcounter{proofstep}{0}
\begin{proofstep}{Martingale properties of $\score_{\run}$}
First, let
\begin{equation}
\payv(\state)
	= - \exof{\nabla\sobj(\state;\sample)} = -\nabla\obj(\state)
\end{equation}
denote the negative gradient of $\obj$ at $\state\in\states$, and write
\begin{equation}
\label{eq:oracle}
\est\payv_{\run}
	= -\nabla\sobj(\act_{\run};\sample_{\run})
\end{equation}
for the corresponding oracle feedback at stage $\run$.
Then, \cref{alg:SMD} may be written in Robbins\textendash Monro form as
\begin{equation}
\label{eq:SA}
\score_{\run+1}
	= \score_{\run} + \step_{\run} \est\payv_{\run}
	= \score_{\run} + \step_{\run} \bracks{\payv(\act_{\run}) + \noise_{\run}},
\end{equation}
where
\begin{equation}
\noise_{\run}
	= \nabla\obj(\act_{\run}) - \nabla\sobj(\act_{\run};\sample_{\run})
\end{equation}
denotes the difference between the mean gradient of $\obj$ at $\act_{\run}$ and the $\run$-th stage gradient sample.%
\footnote{Recall here that there is a one-step offset between $\act_{\run}$ and $\sample_{\run+1}$ at the $\run$-th iteration of \ac{SMD}.}
\begin{subequations}
\label{eq:MDS}
By construction, $\noise_{\run}$ is a \acl{MDS} relative to the history (natural filtration) $\filter_{\run} = \sigma(\sample_{1},\dotsc,\sample_{\run})$ of $\sample_{\run}$, \ie
\begin{equation}
\label{eq:zero-mean}
\exof{\noise_{\run}\given\filter_{\run-1}}
	= 0
	\quad
	\text{for all $\run$.}
\end{equation}
Furthermore, by \cref{asm:L2}, it readily follows that $\noise_{\run}$ has uniformly bounded second moments, \ie there exists some finite $\noisedev\geq0$ such that
\begin{equation}
\label{eq:MSE}
\exof{\dnorm{\noise_{\run}}^{2} \given \filter_{\run-1}}
	\leq \noisevar
	\quad
	\text{for all $\run$},
\end{equation}
implying in turn that $\noise_{\run}$ is bounded in $L^{2}$ (for a more detailed treatment, see \cref{app:proofs}).
\end{subequations}
\end{proofstep}

\begin{proofstep}{Recurrence of $\eps$-neighborhoods}
Invoking the \acl{LLN} for $L^{2}$-bounded \aclp{MDS} and using the Fenchel coupling as an energy function (cf.~\cref{app:proofs}), we show that if $\act_{\run}$ remains outside $\ball(\sols,\eps)$ for sufficiently large $\run$, we must also have $\fench(\sols,\score_{\run})\to-\infty$ \as.
This contradicts the positive-definiteness of $\fench$, so $\act_{\run}$ must enter $\ball(\sols,\eps)$ infinitely often \as.
\end{proofstep}

\begin{proofstep}{Recurrence of Fenchel zones}
By reciprocity (\cref{asm:reciprocity}), $\fball(\sols, \delta)$ always contains an $\eps$-neighborhood of $\sols$.
Since $\act_{\run}$ enters $\ball(\sols,\eps)$ infinitely many times \as, the same must hold for $\fball(\sols, \delta)$.
Our claim and \cref{cor:recurrence} then follow immediately.
\end{proofstep}

\section{Global convergence under coherence}
\label{sec:global}

The convergence of a subsequence of $\act_{\run}$ to the minimum set of \eqref{eq:opt} is one of the crucial steps in establishing our first main result:

\begin{theorem}[Almost sure global convergence]
\label{thm:global}
Suppose that \eqref{eq:opt} is variationally coherent.
Then, under \cref{asm:diff,asm:L2,asm:reciprocity}, $\act_{\run}$ converges with probability $1$ to a \textpar{possibly random} minimum point of \eqref{eq:opt}.
\end{theorem}

\begin{corollary}
\label{cor:global}
If $\obj$ is a non-degenerate quasi-convex \textpar{or pseudo-convex, or convex} function and \cref{asm:diff,asm:L2,asm:reciprocity} hold, the last iterate of \eqref{eq:SMD} converges with probability $1$ to a \textpar{possibly random} minimum point of \eqref{eq:opt}.
\end{corollary}

Before discussing the proof of \cref{thm:global}, it is important to note that most of the literature surrounding \eqref{eq:SMD} and its variants (see \eg \cite{NJLS09,Nes09,Xia10,DAJJ12} and references therein) focuses on the so-called \emph{ergodic average} of $\act_{\run}$, \ie
\begin{equation}
\label{eq:ergodic}
\bar\act_{\run}
	= \frac{\sum_{\runalt=\start}^{\run} \step_{\runalt}\act_{\runalt}}{\sum_{\runalt=\start}^{\run} \step_{\runalt}}
\end{equation}
Despite the appealing ``self-averaging'' properties of $\bar\act_{\run}$ in convex problems \cite{NJLS09,Nes09}, it is not clear how to extend the standard tools used to establish convergence of $\bar\act_{\run}$ beyond convex/monotone problems (even to pseudo-convex programs).
Since convergence of $\act_{\run}$ automatically implies that of $\bar\act_{\run}$, \cref{thm:global} simultaneously establishes the convergence of the last iterate of \ac{SMD} and extends existing ergodic convergence results to a wider class of non-convex stochastic programs.

\cref{cor:global} also extends the corresponding results of \cite{NL14} for the convergence of the last iterate of \eqref{eq:SMD} when $\obj$ is (strongly) convex and $\hreg$ has Lipschitz-continuous gradients (so the induced Bregman divergence can be bounded from above by a quadratic surrogate of the primal norm).
Our proof strategy is similar and relies on the following lemma, often attributed to Gladyshev \cite[p.~49]{Pol87}:%
\footnote{We thank an anonymous reviewer for suggesting this approach.
Our original proof strategy relied on the so-called ``\acs{ODE} method'' of stochastic approximation \cite{Ben99} and was considerably more intricate.}

\begin{lemma}[Gladyshev]
\label{lem:Gladyshev}
Let $a_{\run}$, $\run=\running$, be a sequence of nonnegative random variables such that
\begin{equation}
\label{eq:Fejer}
\exof{a_{\run+1} \given a_{\start},\dotsc,a_{\run}}
	\leq (1 + \delta_{\run}) a_{\run} + \eps_{\run},
\end{equation}
where $\delta_{\run}$ and $\eps_{\run}$ are nonnegative deterministic sequences with
\begin{equation}
\sum_{\run=\start}^{\infty} \delta_{\run} < \infty
	\quad
	\text{and}
	\quad
\sum_{\run=\start}^{\infty} \eps_{\run} < \infty.
\end{equation}
Then, $a_{\run}$ converges \as to some random variable $a_{\infty}\geq0$.
\end{lemma}

As shown below, this ``quasi-Fejérian'' monotonicity property plays a critical part in establishing the convergence of \eqref{eq:SMD}:

\begin{proof}[Proof of \cref{thm:global}]
Let $\sol\in\sols$ be a minimum point of \eqref{eq:opt}.
Then, letting $\fench_{\run} = \fench(\sol,\score_{\run})$, \cref{lem:Fenchel} gives
\begin{flalign}
\label{eq:Fench-bound1}
\fench_{\run+1}
	= \fench(\sol,\score_{\run+1})
	&= \fench(\sol,\score_{\run} + \step_{\run}\est\payv_{\run})
	\notag\\
	&\leq \fench(\sol,\score_{\run})
	+ \step_{\run} \braket{\est\payv_{\run}}{\act_{\run} - \sol}
	+ \frac{\step_{\run}^{2}}{2\strong} \dnorm{\est\payv_{\run}}^{2}
	\notag\\
	&= \fench_{\run}
	+ \step_{\run} \braket{\payv(\act_{\run})}{\act_{\run} - \sol}
	+ \step_{\run} \snoise_{\run}
	+ \frac{\step_{\run}^{2}}{2\strong} \dnorm{\est\payv_{\run}}^{2}
	\notag\\
	&\leq \fench_{\run}
	+ \step_{\run} \snoise_{\run}
	+ \frac{\step_{\run}^{2}}{2\strong} \dnorm{\est\payv_{\run}}^{2},
\end{flalign}
where we set $\snoise_{\run} = \braket{\noise_{\run}}{\act_{\run} - \sol}$ in the third line and used the fact that $\obj$ satisfies \eqref{eq:VC} in the last one.
Since $\score_{\run}$ is predictable relative to $\filter_{\run}$ (\ie $\score_{\run}$ is $\filter_{\run-1}$-measurable), the process $\fench_{\run} = \fench(\sol,\score_{\run})$ is itself adapted to the shifted filtration $\alt\filter_{\run} = \sigma(\sample_{\start},\score_{2}\dotsc,\sample_{\run-1},\score_{\run}) = \filter_{\run-1}$.
Thus, taking conditional expectations and invoking \cref{asm:L2}, the bound \eqref{eq:Fench-bound1} becomes:
\begin{flalign}
\label{eq:Fench-bound-cond}
\exof{\fench_{\run+1} \given \alt\filter_{\run}}
	&\leq \fench_{\run}
	+ \exof{\snoise_{\run} \given \alt\filter_{\run}}
	+ \frac{\step_{\run}^{2}}{2\strong} \exof{\dnorm{\est\payv_{\run}}^{2} \given \alt\filter_{\run}}
	\notag\\
	&= \fench_{\run}
	+ \exof{\snoise_{\run} \given \filter_{\run-1}}
	+ \frac{\step_{\run}^{2}}{2\strong} \exof{\dnorm{\nabla\sobj(\act_{\run};\sample_{\run})}^{2} \given \filter_{\run-1}}
	\notag\\
	&\leq \fench_{\run}
	+ \frac{\step_{\run}^{2} \vbound^{2}}{2\strong},
\end{flalign}
where, in the last line, we used \cref{asm:L2} and the fact that $\noise_{\run}$ is a \acl{MDS} (so $\exof{\snoise_{\run} \given \filter_{\run-1}} = 0$; for a detailed derivation, see the proof of \cref{prop:recurrence} in \cref{app:proofs}).
Hence, with $\sum_{\run=\start}^{\infty} \step_{\run}^{2} < \infty$, \cref{lem:Gladyshev} implies that $\fench_{\run}$ converges \as to some finite limit $\fench_{\infty}$.

Now, by \cref{prop:recurrence}, there exists \as a subsequence $\score_{\run_{\runalt}}$ of $\score_{\run}$ and some (possibly random) $\sol\in\sols$ such that $\lim_{\runalt\to\infty} \fench(\sol,\score_{\run_{\runalt}}) = 0$.
Since the limit $\lim_{\run\to\infty} \fench(\sol,\score_{\run})$ exists \as, it follows that $\lim_{\run\to\infty} \fench(\sol,\score_{\run}) = 0$.
This shows that, with probability $1$, $\act_{\run} = \mirror(\score_{\run})$ converges to some (random) minimum point $\sol$ of \eqref{eq:opt}, as claimed.
\end{proof}

In closing this section, we should note that the conclusion of \cref{thm:global} also applies to problems that are ``almost'' coherent in the sense of \cref{ex:star}, \ie
\begin{enumerate}
[\indent\itshape a\upshape)]
\item
There exists a minimizer $\base\in\sols$ such that $\braket{\nabla\obj(\state)}{\state - \base} \geq 0$ with equality only if $\state\in\sols$.
\item
For all $\sol\in\sols$, $\braket{\nabla\obj(\state)}{\state - \sol} \geq 0$ whenever $\state$ is close enough to $\sol$.
\end{enumerate}
Proving this more general result requires some of the machinery presented in the following section, so we relegate its discussion until all the requisite tools are in place.

\section{Convergence under local/weak coherence}
\label{sec:local}

In this section, our goal is to extend the convergence analysis of the previous section to account for optimization problems that are only ``locally'' coherent.
Building on \cref{def:VC}, these are defined as follows:

\begin{definition}
\label{def:VC-local}
Let $\region$ be a closed set of local minimizers of $\obj$, viz. $\obj(\state) \geq \obj(\sol)$ for all $\sol\in\region$ and all $\state$ sufficiently close to $\region$.
We say that $\region$ is \emph{locally coherent} if there exists an open neighborhood $U$ of $\region$ such that
\begin{equation}
\label{eq:LVC}
\tag{LVC}
\braket{\nabla\obj(\state)}{\state - \sol}
	\geq 0
	\quad
	\text{for all $\state\in U$, $\sol\in\region$},
\end{equation}
and there exists some $\sol\in\region$ such that equality holds in \eqref{eq:LVC} only if $\state\in\region$.
\end{definition} 

An immediate consequence of \cref{def:VC-local} is that locally coherent sets are isolated components of local minimizers of $\obj$.
To see this, if $\region$, $U$ and $\sol$ are as in \cref{def:VC-local} and $\state\in U$ is a local minimizer of $\obj$, we would have $\braket{\nabla\obj(\state)}{z} \geq 0$ for all tangent $z\in\tcone(\state)$.
Applying this to $z = \sol - \state$ gives $\braket{\nabla\obj(\state)}{\state - \sol} \geq 0$, so, by the definition of local coherence, we conclude that $\state\in\region$.

We also note that although the minimum set of a globally coherent problem is \emph{a fortiori} locally coherent, the converse need not hold.
A concrete example of a function which is not globally coherent but which admits a locally coherent minimum is the Rosenbrock test function
\begin{equation}
\label{eq:Rosenbrock}
\obj(\state)
	= \sum_{i=1}^{\vdim} \bracks{100(\state_{i+1} - \state_{i})^{2} + (1 - \state_{i}^{2})},
	\quad
	\state\in[-2,2]^{\vdim},
\end{equation}
which has seen extensive use in the literature as a non-convex convergence speed benchmark (cf.~\cref{sec:numerics}).%
\footnote{Local coherence can be proved by a straightforward algebraic calculation (omitted for concision).}
From this example, we see that the profile of $\obj$ around a locally coherent set could be highly non-convex, possibly including a wide variety of valleys, talwegs and ridges;
in fact, even \emph{quasi}-convexity may fail to hold locally.

Now, in contrast to globally coherent optimization problems, an ``unlucky'' gradient sample could drive \eqref{eq:SMD} out of the ``basin of attraction'' of a locally coherent set (the largest neighborhood $U$ for which \eqref{eq:LVC} holds), possibly never to return.
For this reason, instead of focusing on global convergence results with probability one, we will focus on local convergence with high probability.
Our main result along these lines is as follows:

\begin{theorem}[Local convergence with high probability]
\label{thm:local}
Let $\region$ be locally coherent for \eqref{eq:opt} and fix some confidence level $\delta>0$.
Then, under \cref{asm:diff,asm:L2,asm:reciprocity}, there exists an open neighborhood $\nhd$ of $\region$, independent of $\delta$, such that
\begin{equation}
\label{eq:conv-prob}
\probof{\textup{$\act_{\run}$ converges to $\region$} \given \act_{\start}\in \nhd}
	\geq 1-\delta,
\end{equation}
provided that the algorithm's step-size sequence $\step_{\run}$ is small enough.
\end{theorem}

%
%

\begin{remark*}
As a concrete application of \cref{thm:local}, fix any $\beta \in (1/2,1]$.
Then, for every confidence level $\delta>0$, \cref{thm:local} implies that there exists some small enough $\step>0$ such that if \cref{alg:SMD} is run with step-size $\step_{\run} = \step/\run^{\beta}$, \cref{eq:conv-prob} holds.
We emphasize the interesting point here:
the open neighborhood $\nhd$ is fixed once and for all, and does not depend on the probability threshold $\delta$.
That is, to get convergence with higher probability, it is \emph{not} necessary to assume that $\act_{\start}$ starts closer to $\region$:
one need only use a smaller step-size sequence satisfying \eqref{eq:L2L1}.
\end{remark*}
\vspace{1ex}

The key idea behind the proof of \cref{thm:local} is as follows:
First, it suffices to consider the case where $\region$ consists of a single local minimizer $\sol$;
the argument for the general follows the same techniques as in \cref{sec:global}.
Then, conditioning on the event that $\act_{\run}$ remains sufficiently close to $\sol$ for all $\run$, convergence can be obtained by invoking \cref{thm:global} and treating \eqref{eq:opt} as a variationally coherent problem over a smaller subset of $\feas$ over which \eqref{eq:LVC} holds.
Therefore, to prove \cref{thm:local}, it suffices to show that $\act_{\run}$ remains close to $\sol$ for all $\run$ with probability no less than $1-\delta$.
To achieve this, we rely again on the properties of the Fenchel coupling, and we decompose the stochastic errors affecting each iteration of the algorithm into a first-order $\bigoh(\step_{\run})$ martingale term and a second-order $\bigoh(\step_{\run}^{2})$ submartingale perturbation.
Using Doob's maximal inequality, we then show that the aggregation of both errors remains controllably small with probability at least $1-\delta$.
\vspace{1ex}

We formalize all this below:

\setcounter{proofstep}{0}
\begin{proofof}{Proof of \cref{thm:local}}
We break the proof into three steps.

\begin{proofstep}{Controlling the martingale error}
Fix some $\epsilon > 0$.
As in the proof of \cref{thm:global}, let $\noise_{\run} = \nabla\obj(\act_{\run}) - \nabla\sobj(\act_{\run};\sample_{\run})$ and set $\snoise_{\run} = \braket{\noise_{\run}}{\act_{\run}-\sol}$ where $\sol\in\region$ is such that \eqref{eq:LVC} holds as an equality only if $\state\in\region$ (cf.~\cref{def:VC-local}).
We show below that there exists a step-size sequence $(\step_{\run})_{\run=\start}^{\infty}$ such that
\begin{equation}
\label{eq:mg_noise_bound}
\probof*{\sup_{\run} \sum_{\runalt=\start}^{\run} \step_{\runalt} \snoise_{\runalt} \leq \eps} \geq 1 - \frac{\delta}{2}.
\end{equation}
To show this, we start by noting that, as in the proof of \cref{prop:recurrence}, the aggregate process $S_{\run} = \sum_{\runalt=\start}^{\run} \step_{\runalt} \snoise_{\runalt}$ is a martingale relative to the natural filtration $\filter_{\run}$ of $\sample_{\run}$.
Then, letting $R= \sup_{\state\in\feas} \norm{\state}$, we can bound the variance of each individual term of $S_{\run}$ as follows:
\begin{flalign}
\label{eq:MSE_local1}
\exof{\snoise_{\runalt}^{2}}
	&= \exof{\exof{\abs{\braket{\noise_{\runalt}}{\act_{\runalt} - \sol}}^{2} \given \filter_{\runalt-1}}}
	\notag\\
	&\leq \exof{\exof{\dnorm{\noise_{\runalt}}^{2} \norm{\act_{\runalt} - \sol}^{2} \given \filter_{\runalt-1}}}
	\notag\\
	&= \exof{\norm{\act_{\runalt} - \sol}^{2} \exof{\dnorm{\noise_{\runalt}}^{2} \given \filter_{\runalt-1}}}
	\notag\\
	&\leq R^{2} \vbound^{2},
\end{flalign}
where the first inequality follows from the definition of the dual norm
and the second one follows from \eqref{eq:MSE}. Consequently, by Doob's maximal inequality (\cref{thm:mg_convergence4} in \cref{app:aux}), we have:
\begin{equation}
\label{eq:MSE_local2}
\probof*{\sup_{0\leq \runalt \leq \run} S_{\runalt} \geq \epsilon}
	\leq \probof*{\sup_{0\leq \runalt \leq \run} \abs{S_{\runalt}} \geq \epsilon}
	\leq \frac{\exof{S_{\run}^{2}}}{\eps^{2}}
	\leq \frac{R^{2} \vbound^{2} \sum_{\runalt=\start}^{\run} \step_{\runalt}^{2}}{\eps^{2}},
	\end{equation}
where the last inequality follows from expanding $\exof{\abs{S_{\run}}^{2}}$, using \cref{eq:MSE_local1}, and noting that $\exof{\snoise_{\runalt}\snoise_{\ell}} = \exof{\exof{\snoise_{\runalt} \snoise_{\ell}} \given \filter_{\runalt\vee\ell-1}} = 0$ whenever $\runalt\neq\ell$.
Therefore, by picking $\step_{\run}$ so that $\sum_{\runalt=\start}^{\infty} \step_{\runalt}^{2} \leq \epsilon^{2}\delta/(2R^{2}\vbound^{2})$, \cref{eq:MSE_local2} gives
\begin{equation}
\label{eq:MSE_local3}
\probof*{\sup_{0\leq k \leq t} S_{\runalt} \geq \epsilon}
	\leq \frac{R^{2} \vbound^{2} \sum_{\runalt=\start}^{\run} \step_{\runalt}^{2}}{\eps^{2}}
	\leq \frac{R^{2} \vbound^{2} \sum_{\runalt=\start}^{\infty} \step_{\runalt}^{2}}{\eps^{2}}
	\leq \frac{\delta}{2}
	\quad
	\text{for all $\run$}.
\end{equation}
Since the above holds for all $\run$, our assertion follows.
\end{proofstep}

\begin{proofstep}{Controlling the submartingale error}
Again, fix some $\epsilon > 0$ and, with a fair amount of foresight, let $R_{\run} = (2\strong)^{-1} \sum_{\runalt=\start}^{\run} \step_{\runalt}^{2} \dnorm{\est\payv_{\runalt}}^{2}$.
By construction, $R_{\run}$ is a non-negative submartingale relative to $\filter_{\run}$.
We again establish that there exists step-size sequence $(\step_{\run})_{\run=\start}^{\infty}$ satisfying the summability condition \eqref{eq:L2L1} and such that
\begin{equation}
\label{eq:smg_noise_bound}
\probof*{\sup_{\run} R_{\run} \leq \eps}
	\geq 1 - \frac{\delta}{2}.
\end{equation}
To show this, Doob's maximal inequality for submartingales (\cref{thm:mg_convergence3}) gives
\begin{equation}
\label{eq:prob-bound2}
\probof*{\sup_{0 \leq \runalt \leq \run} R_{\runalt} \geq \eps} 
	\leq \frac{\exof{R_{\run}}}{\eps}
	\leq \frac{\vbound^{2} \sum_{\runalt=\start}^{\run} \step_{\runalt}^{2}}{2\strong\eps},
	\end{equation}
where we used the fact that $\exof{\cramped{\dnorm{\nabla\sobj(\act_{\run};\sample_{\run})}^{2}}} \leq \vbound^{2}$ for some finite $\vbound<\infty$.
Consequently, if we choose $\step_{\run}$ so that $\sum_{\runalt=\start}^{\infty} \step_{\runalt}^{2} \leq \strong\delta\eps/\vbound^{2}$, \cref{eq:prob-bound2} readily gives
\begin{equation}
\probof*{\sup_{0 \leq \runalt \leq \run} R_{\runalt} \geq \eps}
	\leq \frac{\vbound^{2} \sum_{\runalt=\start}^{\infty} \step_{\runalt}^{2}}{2\strong\eps}
	\leq \frac{\delta}{2}
	\quad
	\text{for all $\run$}
\end{equation}
Since the above is true for all $\run$, \cref{eq:smg_noise_bound} follows.
\end{proofstep}

\begin{proofstep}{Error aggregation}
To combine the above, fix some sufficiently small $\bar\eps>0$ so that $\fball(\sol, 3\bar\eps) \subset U$, where $U$ is the open neighborhood  given in \ref{eq:LVC}.
Furthermore, let $\nhd = \fball(\sol, \bar\eps)$ and pick a step-size sequence $\step_{\run}$ satisfying \eqref{eq:L2L1} and such that
\begin{equation}
\label{eq:step-small}
\sum_{\run=\start}^{\infty} \step_{\run}^{2}
	\leq \min\braces*{\frac{\delta\bar\eps^{2}}{2R^{2}\vbound^{2}}, \frac{\strong\delta\bar\eps}{\vbound^{2}}}.
\end{equation}
If $\act_{\start}\in\nhd$, it follows that $\fench(\sol,\score_{\start}) < \bar\eps$ by the definition of $\fball$ (cf.~\cref{def:metrics}).
Then, by \cref{eq:mg_noise_bound,eq:smg_noise_bound}, we get $\probof{\sup_{\run} S_{\run} \geq \bar\eps} \leq \delta/2$ and $\probof{\sup_{\run} R_{\run} \geq \bar\eps} \leq \delta/2$.
Consequently, with this choice of $\step_{\run}$, it follows that
\begin{equation}
\probof*{\sup\nolimits_{\run} \max\{S_{\run},R_{\run}\} \leq \bar\eps}
	\geq 1 - \delta/2 - \delta/2
	= 1 - \delta
\end{equation}
Then, letting $\fench_{\run} = \fench(\sol,\score_{\run})$ and arguing as in the proof of \cref{thm:global}, we may expand $\fench_{\run} = \fench(\sol,\score_{\run})$ to get
\begin{flalign}
\label{eq:Dbound-stoch}
\fench_{\run}
	&= \fench(\sol, \score_{\run} + \step_{\run}\est\payv_{\run})
	\notag\\
	&\leq \fench_{\run}
	+ \step_{\run} \braket{\payv(\act_{\run})}{\act_{\run} - \sol}
	+ \step_{\run} \snoise_{\run}
	+ \frac{\step_{\run}^{2}}{2\strong} \dnorm{\nabla\sobj(\act_{\run};\sample_{\run})}^{2}
\end{flalign}
with $\snoise_{\run} = \braket{\noise_{\run}}{\act_{\run}-\sol}$ defined as above.
Telescoping \eqref{eq:Dbound-stoch} then yields
\begin{flalign}
\label{eq:Dbound-n}
\fench_{\run}
	&\leq \fench_{\start}
	+ \sum_{\runalt=\start}^{\run} \step_{\runalt} \braket{\payv(\act_{\runalt})}{\act_{\runalt} - \sol}
	+ S_{\run}
	+ R_{\run}\\
	&\leq \bar\eps + \sum_{\runalt=\start}^{\run} \step_{\runalt}\braket{\payv(\act_{\runalt})}{\act_{\runalt} - \sol} + \bar\eps + \bar\eps,	\end{flalign}
with probability at least $1-\delta$.
Therefore, with probability at least $1-\delta$, we have
\begin{equation}
\label{eq:Fench-maximal}
\fench(\sol,\score_{\run})
	\leq 3\bar\eps + \sum_{\runalt=\start}^{\run} \step_{\runalt}\braket{\payv(\act_{\runalt})}{\act_{\runalt} - \sol}.
\end{equation}

Now, assume inductively that, for all $\runalt\leq\run$, we have $\fench(\sol,\score_{\runalt}) \leq 3\bar\eps$ or, equivalently, $\act_{\runalt} \in \fball(\sol, 3\bar\eps)$.
In turn, this implies that $\braket{\payv(\act_{\runalt})}{\act_{\runalt} - \sol} \leq 0$ for all $\runalt\leq \run$, and hence, by \eqref{eq:Fench-maximal}, that $\fench(\sol,\score_{\run}) \leq 3\bar\eps$ as well.
Since the base case $\act_{\start} \in \nhd = \fball(\sol, \bar\eps) \subset \fball(\sol, 3\bar\eps)$ is satisfied automatically, we conclude that $\act_{\run}$ stays in $\fball(\sol, 3\bar\eps) \subset U$ for all $\run$ with probability at least $1-\delta$.
Our claim then follows by conditioning on this event and repeating the same steps as in the proof of \cref{thm:global}.
\end{proofstep}	
\end{proofof}

We close this section by revisiting the notion of weak coherence (\cref{def:VC-weak}).
In view of \cref{def:VC-local}, we see that weak coherence mixes elements of both global \emph{and} local coherence:
on the one hand, it posits the existence of a (global) minimizer $\base\in\sols$ for which \eqref{eq:VC} holds globally, thus satisfying the second part of \cref{def:VC};
on the other hand, minimizers other than $\base$ are only required to satisfy \eqref{eq:VC} locally (though they need not be locally coherent themselves).
From a stability viewpoint, this means that individual elements of a weakly coherent set may be stable but not necessarily attracting (even locally).
However, taken as a whole, weakly coherent sets are \emph{globally} attracting:

\begin{theorem}
\label{thm:weak}
Suppose that \eqref{eq:opt} is weakly coherent.
Then, under \cref{asm:diff,asm:L2,asm:reciprocity}, $\act_{\run}$ converges with probability $1$ to a \textpar{possibly random} minimum point of \eqref{eq:opt}.
\end{theorem}

\begin{proof}
The proof is essentially a combination of the proofs of \cref{thm:global,thm:local}, so we only provide the main arguments and omit the minor details.

The first observation is that the conclusion of \cref{prop:recurrence} only requires the first part of \cref{def:VC-weak} (simply take $\sol=\base$ in the proof of \cref{prop:recurrence}).
From this, we conclude that the generated sequence $\act_{\run}$ admits \as a subsequence $\act_{\run_{\runalt}}$ converging to some $\sol\in\sols$.

In view of this, arguing as in the proof of \cref{thm:local} allows us to conclude that, with probability $1$, $\act_{\run}$ itself remains in some neighborhood $\nhd$ of $\sol$ such that $\braket{\nabla\obj(\state)}{\state-\sol} \geq 0$ for all $\state\in\nhd$.
To see this, recall first that $\sum_{\run=\start}^{\infty}\step_{\run}^{2}<\infty$;
thus, by starting at some sufficiently large $\run_{0}$ (recall that $\act_{\run}$ comes arbitrarily close to $\sol$ infinitely many times), we can assume without loss of generality that \eqref{eq:step-small} is satisfied for any $\eps,\delta>0$.
Since $\delta$ is arbitrary, this means that \eqref{eq:Fench-maximal} actually holds with probability $1$, and our claim follows.

The above shows that $\braket{\nabla\obj(\act_{\run})}{\act_{\run} - \sol}$ for all sufficiently large $\run$.
Thus, going back to the proof of \cref{thm:global}, it follows that the Fenchel coupling $\fench_{\run} = \fench(\sol,\score_{\run})$ converges.
Since $\liminf_{\run\to\infty} \fench_{\run} = 0$ (by \cref{asm:reciprocity} and the fact that $\act_{\run_{\runalt}} = \mirror(\score_{\run_{\runalt}}) \to \sol$), we conclude that $\lim_{\run\to\infty} \act_{\run} = \sol$, and our proof is complete.
\end{proof}

\section{Sharp minima and applications}
\label{sec:sharp}

Given the randomness involved at each step, obtaining an almost sure (or high probability) bound for the convergence speed of the last iterate of \ac{SMD} is fairly involved:
indeed, in contrast to the ergodic rate analysis of \ac{SMD} for convex programs, there is no intrinsic averaging in the algorithm's last iterate, so it does not seem possible to derive a precise black box convergence rate for $\act_{\run}$.
Essentially, as in the analysis of \cref{sec:local}, a single ``unlucky'' gradient sample could violate any convergence speed estimate that is probabilistically independent of any finite subset of realizations.

Despite this difficulty, if \ac{SMD} is run with a \emph{surjective} mirror map, we show below that $\act_{\run}$ reaches a minimum point of \eqref{eq:opt} in a \emph{finite} number of iterations
for a large class of optimization problems that admit \emph{sharp} minima (see below).
As we noted in \cref{sec:setup}, an important example of a surjective mirror map is the standard Euclidean projection $\Eucl(\dstate) = \argmin_{\state\in\states} \norm{\dstate-x}_{2}$.
The resulting descent method is the well-known \acdef{SGD} algorithm (cf.~\cref{alg:SGD} below), so our results in this section also provide new insights into the behavior of \ac{SGD}.

\subsection{Definition and characterization}
\label{subsec:sharp}

The starting point of our analysis is Polyak's fundamental notion of a \emph{sharp minimum} \cite[Chapter~5.2]{Pol87}, which describes functions that grow at least linearly around their minimum points:

\begin{definition}
\label{def:sharp}
We say that $\sol\in\states$ is a \emph{$\sharp$-sharp} (local) minimum of $\obj$ if
\begin{equation}
\label{eq:sharp}
\obj(\state)
	\geq \obj(\sol) + \sharp \norm{\state - \sol}
	\quad
	\text{for some $\sharp>0$ and all $\state$ sufficiently close to $\sol$}.
\end{equation}
\end{definition}

Polyak's original definition concerned global sharp minima of unconstrained (convex) optimization problems;
by contrast, the above definition is tailored to local optima of constrained (and possibly non-convex) programs.
In particular, \cref{def:sharp} implies that sharp minima are isolated (local) minimizers of $\obj$, and they remain invariant under small perturbations of $\obj$ (assuming of course that such a minimizer exists in the first place).
In what follows, we shall omit the modifier ``local'' for concision and rely on the context to resolve any ambiguities.

Sharp minima admit a useful geometric interpretation in terms of the polar cone of $\states$.
To state it, recall first the following basic facts from convex analysis:

\begin{definition}
\label{def:cone}
Let $\states$ be a closed convex subset of $\R^{\vdim}$.
Then:
\begin{enumerate}
\item
The \emph{tangent cone} $\tcone(\base)$ to $\states$ at $\base$ is defined as the closure of the set of all rays emanating from $\base$ and intersecting $\states$ in at least one other point.
\item
The \emph{dual cone} $\dcone(\base)$ to $\states$ at $\base$ is the dual set of $\tcone(\base)$, viz. $\dcone(\base) = \setdef{\dstate\in\R^{\vdim}}{\braket{\dstate}{z}\geq0 \; \text{for all $z\in\tcone(\base)$}}$.
\item
The \emph{polar cone} $\pcone(\base)$ to $\states$ at $\base$ is the polar set of $\tcone(\base)$, viz. $\pcone(\base) = - \dcone(\base) = \setdef{\dstate\in\R^{\vdim}}{\braket{\dstate}{z}\leq0\;\text{for all $z\in\tcone(\base)$}}$.
\end{enumerate}
\end{definition} 
\vspace{1ex}

The above gives the following geometric characterization of sharp minima:

\begin{lemma}
\label{lem:sharp}
If $\sol\in\states$ is a $\sharp$-sharp minimum of $\obj$, we have
\begin{equation}
\label{eq:sharp-tangent}
\braket{\nabla\obj(\sol)}{z}
	\geq \sharp\norm{z}
	\quad
	\text{for all $z\in\tcone(\sol)$}.
\end{equation}
In particular, $\nabla\obj(\sol)$ belongs to the topological interior of $\dcone(\state)$.
Conversely, if \eqref{eq:sharp-tangent} holds and $\obj$ is convex, $\sol$ is sharp.
\end{lemma}

%
%

\begin{proofof}{Proof of \cref{lem:sharp}}
For the direct implication, fix some $\state\in\states$ satisfying \eqref{eq:sharp}, and let $z = \state - \sol \in\tcone(\sol)$.
Then, by the definition of a sharp minimum, we get
\begin{equation}
\obj(\sol + \tau z)
	\geq \obj(\sol) + \sharp\tau\norm{z}
	\quad
	\text{for all $\tau\in[0,1]$}.
\end{equation}
In turn, this implies that
\begin{equation}
\frac{\obj(\sol + tz) - \obj(\sol)}{t}
	\geq \sharp \norm{z}
	\quad
	\text{for all sufficiently small $t>0$}.
\end{equation}
Hence, taking the limit $\tau\to0^{+}$, we get $\braket{\nabla\obj(\sol)}{z} \geq \sharp\norm{z}$, and our claim follows from the definition of $\tcone(\sol)$ as the closure of the set of all rays emanating from $\sol$ and intersecting $\states$ in at least one other point.
As for the converse implication, simply note that $\obj(\state) - \obj(\sol) \geq \braket{\nabla\obj(\sol)}{\state - \sol} \geq \sharp\norm{\state - \sol}$ if $\obj$ is convex.
\end{proofof}

\begin{example}[Linear programs]
\label{ex:linear}
A first important class of examples of functions that admit sharp minima is that of generic linear programs.%
\footnote{``Generic'' means here that $\states$ is a polytope, $\obj\from\states\to\R$ is affine, and $\obj$ is constant only on the zero-dimensional faces of $\states$.
Any linear program can be turned into a generic one after an arbitrarily small perturbation.}
Indeed, by definition, a linear function grows (exactly) linearly around its minimum points so, by genericity,
\end{example}
\vspace{1ex}

\begin{example}[Concave minimization]
\label{ex:concave}
For a non-convex class of examples, let $\obj\from\states\to\R$ be a strictly \emph{concave} function defined over a convex polytope $\states$ of $\R^{\vdim}$.
Concavity implies that $\obj$ is superharmonic, \ie
\begin{equation}
\label{eq:super}
\Delta\obj(\state)
	= \sum_{i=1}^{\vdim} \frac{\pd^{2}\obj}{\pd \state_{i}^{2}} \leq 0
\end{equation}
for all $\state\in\states$.%
\footnote{We tacitly assume above that $\obj$ is twice-differentiable but this conclusion still holds even if $\obj$ is not differentiable.}
By the minimum principle for superharmonic functions, the minimum of $\obj$ over any connected region $\region$ of $\states$ is attained at the boundary of $\region$.
Hence, by strict concavity, we conclude that the local minima of $\obj$ are attained at $0$-dimensional faces of $\states$, and they are \emph{de facto} sharp (simply note that $\obj$ is strictly concave along any ray of the form $\sol+tz$, $z\in\tcone(\sol)$).
\end{example}
\vspace{1ex}

Sharp minima have several other interesting and useful properties.
First, by \cref{lem:sharp}, sharp minimum points are locally coherent.
To see this, simply note that for all $\state \in \states$ sufficiently close to $\sol$ (with $\state \neq \sol$), we have $z = \state - \sol \in \tcone(\sol)$ and $\braket{\nabla\obj(\sol)}{z} \geq \sharp\norm{z} > 0$.
Consequently, $\langle \nabla \obj(\sol), \state-\sol\rangle > 0$, implying by continuity that $\braket{\nabla\obj(\state)}{\state-\sol} > 0$ for all $\state$ in some open neighborhood of $\sol$ (excluding $\sol$).
In addition, 
if \eqref{eq:opt} is variationally coherent, then a sharp (local) minimum is globally sharp as well.

A second important property is that the dual cone $\dcone(\sol)$ of a sharp minimum must necessarily have nonempty topological interior \textendash\ since it contains $\nabla\obj(\sol)$ by \cref{lem:sharp}.
This implies that sharp minima can only occur at \emph{corners} of $\states$:
for instance, if a sharp minimum were an interior point of $\states$, the dual cone to $\states$ at $\sol$ would be a proper linear subspace of the ambient vector space, so it would have no topological content (see also \cref{ex:concave} above).

\subsection{Global convergence in a finite number of iterations}
\label{subsec:finite-global}

We now turn to showing that, if a variationally coherent program admits a sharp minimum $\sol$, \cref{alg:SMD} reaches $\sol$ in a finite number of iterations \as.
The interesting feature here is that convergence is guaranteed to occur in a \emph{finite} number of iterations:
specifically, there exists some (random) $\run_{0}$ such that $\act_{\run} = \sol$ for all $\run\geq\run_{0}$.
In general, this is a fairly surprising property for a first-order descent scheme, even if one considers the ergodic average $\run^{-1} \sum_{\runalt=\start}^{\run} \act_{\runalt}$:
a priori, a single ``bad'' sample could kick $\act_{\run}$ away from $\sol$, which is the reason why (ergodic) convergence rates are typically asymptotic.

The key intuition behind our analysis is that sharp minima must occur at corners of $\states$ (as opposed to interior points).
As a further key insight, when the solution of \eqref{eq:opt} occurs at a corner, noisy gradients may still play the role of a random disturbance;
however, since they are applied to the dual process $\score_{\run}$, a surjective mirror map would immediately project $\score_{\run}$ back to a corner of $\states$ if $\score_{\run}$ has progressed far enough in the interior of the polar cone to $\states$ at $\state$.
This ensures that the last iterate $\act_{\run}$ of \ac{SMD} will stay \emph{exactly} at the optimal point, irrespective of the persistent noise entering \cref{alg:SMD}.

Exploiting these insights and the structural properties of sharp minima, we have:

\begin{theorem}
\label{thm:sharp-global}
Suppose that \eqref{eq:opt} is variationally coherent.
If $\obj$ admits a \textpar{necessarily unique} sharp minimum $\sol$, and \cref{alg:SMD} is run with a surjective mirror map and \cref{asm:diff,asm:L2,asm:reciprocity} hold, $\act_{\run}$ converges to $\sol$ in a finite number of steps \as.
More precisely, we have
\begin{equation}
\probof{\textup{$\act_{\run}=\sol$ for all sufficiently large $\run$}}
	= 1
\end{equation}
\end{theorem}

\begin{corollary}
\label{cor:sharp-convex}
Let $\obj$ be a non-degenerate quasi-convex \textpar{or pseudo-convex, or convex} function and let $\sol$ be a sharp minimum of $\obj$.
Then, with assumptions as above, $\act_{\run}$ reaches $\sol$ in a finite number of steps \as.
\end{corollary}

The prototypical example of a surjective mirror map is the Euclidean projector $\Eucl(\dstate) = \argmin_{\state\in\states} \norm{\dstate-\state}_{2}$ induced by the quadratic regularization function $\hreg(\state) = \norm{\state}_{2}^{2}/2$ (cf.~\cref{ex:Eucl}).
The resulting descent scheme is the well-known \acf{SGD} algorithm (see \cref{alg:SGD} for a pseudocode implementation), for which we obtain the following novel convergence result:


\begin{algorithm}[tbp]
\caption{Stochastic gradient descent (\ac{SGD})} 
\label{alg:SGD}

\tt
\begin{algorithmic}[1]
\Require
	step-size sequence $\step_{\run}>0$
\State
	choose $\score\in\R^{\vdim}$, $\act=\Eucl(\score)$
	\Comment{initialization}
\For{$\run=\running$}
	\State
		get $\est\payv = -\nabla\sobj(\act;\sample)$
		\Comment{oracle feedback}
	\State
		set $\score \leftarrow \score + \step_{\run} \est\payv$
		\Comment{gradient step}
	\State
		set $\act \leftarrow \Eucl(\score)$
		\Comment{set state}
\EndFor
\\
\Return $\act$
	\Comment{output}
\end{algorithmic}
\end{algorithm}


\begin{corollary}
\label{cor:linear}
If \eqref{eq:opt} is a generic linear program, the last iterate $\act_{\run}$ of \ac{SGD} reaches its \textpar{necessarily unique} solution in a finite number of steps \as.
\end{corollary}

With all this said and done, we proceed with the proof of \cref{thm:sharp-global}:

\begin{proofof}{Proof of \cref{thm:sharp-global}}
Since $\sol$ is a $\sharp$-sharp minimum, there exists a sufficiently small open neighborhood $\nhd$ of $\sol$ such that $\braket{\nabla\obj(\state)}{z} \geq \sharp\norm{z}/2$ for all $z\in\tcone(\sol)$ and all $\state \in \nhd$.
By \cref{thm:global}, $\act_{\run}$ converges to $\sol$ \as, so there exists some (random) $\run_{0}$ such that $\act_{\run} \in \nhd$ for all $\run\geq\run_{0}$.
In turn, this implies that $\braket{\nabla\obj(\act_{\run})}{z} \geq \sharp\norm{z}/2$ for all $\run\geq \run_{0}$.
Thus,
continuing to use the notation $\payv(\act_{\run}) = -\nabla\obj(\act_{\run})$ and $\noise_{\run} = \nabla\obj(\act_{\run}) -\nabla\sobj(\act_{\run};\sample_{\run})$, we get for all $z\in\tcone(\sol)$ with $\norm{z}\leq1$:
\begin{flalign}
\label{eq:sharp-bound1}
\braket{\score_{\run}}{z}
	&= \braket*{\score_{\run_{0}} + \sum_{\runalt=\run_{0}}^{\run} \step_{\runalt} \est\payv_{\runalt}}{z}
	\notag\\
	&= \braket{\score_{\run_{0}}}{z}
	+ \sum_{\runalt=\run_{0}}^{\run} \step_{\runalt} \braket{\payv(\act_{\runalt})}{z}
	+ \sum_{\runalt=\run_{0}}^{\run} \step_{\runalt} \braket{\noise_{\runalt}}{z}
	\notag\\
	&\leq \dnorm{\score_{\run_{0}}}
	- \frac{\sharp}{2} \sum_{\runalt=\run_{0}}^{\run} \step_{\runalt}
	+ \sum_{\runalt=\run_{0}}^{\run} \step_{\runalt} \braket{\noise_{\runalt}}{z},
\end{flalign}
where, in the last line, we used the fact that $\act_{\runalt}\in\nhd$ for all $\runalt\geq\run_{0}$.

As we discussed in the proof of \cref{thm:global}, $\step_{\run}\noise_{\run}$ is a \acl{MDS} relative to the natural filtration $\filter_{\run}$ of $\sample_{\run}$.
Hence, by the \acl{LLN} for martingale differences (cf.~\cref{thm:mg_convergence1} for $p=2$ and $\tau_{\run} = \sum_{\runalt=0}^{\run} \step_{\runalt}$), we get
\begin{equation}
\lim_{\run\to0} \frac{\sum_{\runalt=\run_{0}}^{\run} \step_{\runalt} \noise_{\runalt}}{\sum_{\runalt=\run_{0}}^{\run} \step_{\runalt}}
	= 0
	\quad
	\as.
\end{equation}
Thus, there exists some $\run^{\ast}$ such that $\dnorm{\sum_{\runalt=\run_{0}}^{\run} \step_{\runalt} \noise_{\runalt}} \leq (\sharp/4) \sum_{\runalt=\run_{0}}^{\run} \step_{\runalt}$ for all $\run\geq\run^{\ast}$ \as.
\cref{eq:sharp-bound1} then implies
\begin{flalign}
\label{eq:sharp-bound2}
\braket{\score_{\run}}{z}
	&\leq  \dnorm{\score_{\run_{0}}}
	- \frac{\sharp}{2} \sum_{\runalt=\run_{0}}^{\run} \step_{\runalt}
	+ \sum_{\runalt=\run_{0}}^{\run} \step_{\runalt} \braket{ \noise_{\runalt}}{z}
	\notag\\
	&\leq \dnorm{\score_{\run_{0}}}
	- \frac{\sharp}{2} \sum_{\runalt=\run_{0}}^{\run} \step_{\runalt} + \frac{\sharp}{4} \sum_{\runalt=\run_{0}}^{\run} \step_{\runalt} = \dnorm{\score_{\run_{0}}} - \frac{\sharp}{4} \sum_{\runalt=\run_{0}}^{\run} \step_{\runalt},
\end{flalign}
where we used the assumption that $\norm{z}\leq1$.
Since $\sum_{\runalt=\run_{0}}^{\run} \step_{\runalt}\to\infty$ as $\run\to\infty$,
we get $\braket{\score_{\run}}{z}\to-\infty$ with probability $1$.

To proceed, we claim that $\dsol + \pcone(\sol)\subseteq\mirror^{-1}(\sol)$ whenever $\mirror(\dsol) = \sol$,
\ie $\mirror^{-1}(\sol)$ contains all cones of the form $\dsol + \pcone(\sol)$ for $\dsol\in\mirror^{-1}(\sol)$.
Indeed, note first that $\sol = \mirror(\dsol)$ if and only if $\dsol\in\pd \hreg(\sol)$, where $\pd \hreg(\sol)$ is the set of all subgradients of $h$ at $\sol$ \cite{Roc70}.
Therefore, it suffices to show that $\dsol + w \in \pd \hreg(\sol)$ whenever $w\in\pcone(\sol)$.
To that end, note that the definition of the polar cone gives
\begin{equation}
\braket{w}{\state-\sol}
	\leq 0
	\quad
	\text{for all $\state\in\states$, $w\in\pcone(\sol)$},
\end{equation} 
and hence
\begin{equation}
\hreg(\state)
	\geq \hreg(\sol) + \braket{\dsol}{\state - \sol}
	\geq \hreg(\sol) + \braket{\dsol + w}{\state - \sol}.
\end{equation}
The above shows that $\dsol +w \in \pd \hreg(\sol)$, as claimed.

With $\mirror$ surjective, the set $\mirror^{-1}(\sol)$ is nonempty, so it suffices to show that $\score_{\run}$ lies in the cone $\dsol + \pcone(\sol)$ for some $\dsol\in\mirror^{-1}(\sol)$ and all sufficiently large $\run$.
To do so, simply note that $\score_{\run} \in \dsol + \pcone(\sol)$ if and only if $\braket{\score_{\run} - \dsol}{z} \leq 0$ for all $z\in\tcone(\sol)$ with $\norm{z}=1$.
Since $\braket{\score_{\run}}{z}$ converges to $-\infty$ \as, our assertion is immediate.
\end{proofof}

\subsection{Local convergence in a finite number of iterations}

Our convergence analysis for locally coherent sets of minimizers (cf.~\cref{sec:local}),
showed that \ac{SMD} converges locally with high probability.
Our last result in this section complements this analysis by showing that, with high probability, \ac{SMD} converges locally to sharp local minima in a \emph{finite} number of iterations:

\begin{theorem}
\label{thm:sharp-local}
Let $\sol$ be a sharp \textpar{local} minimum of $\obj$, and fix some confidence level $\delta>0$.
If \cref{alg:SMD} is run with a surjective mirror map and \cref{asm:diff,asm:L2,asm:reciprocity} hold, there exists an open neighborhood $\nhd$ of $\sol$, independent of $\delta$, such that
\begin{equation}
	\probof{\textup{$\act_{\run}=\sol$ for all sufficiently large $\run$} \given \act_{\start}\in \nhd}
	\geq 1-\delta,
\end{equation}
provided that the algorithm's step-size sequence $\step_{\run}$ is small enough.
\end{theorem}

Given that local minimizers of concave minimization programs are sharp, an application of \cref{thm:sharp-local} gives the following convergence result for \ac{SGD}:

\begin{corollary}
\label{cor:concave}
Suppose that $\obj$ is strictly concave as in \cref{ex:concave}.
Then, under assumptions \cref{asm:diff,asm:L2}, the last iterate of \ac{SGD} converges locally to a local minimum of \eqref{eq:opt} with arbitrarily high probability.
\end{corollary}

\begin{proof}[Proof of \cref{thm:sharp-local}]
Under the stated assumptions, \cref{thm:local} implies that there exists an open neighborhood $\nhd$ of $\sol$ such that $\probof{\lim_{\run\to\infty} \act_{\run} = \sol \given \act_{0} \in \nhd} \geq 1-\delta$.
In turn, this means that there exists some (random) $\run_{0}$ which is finite with probability at least $1-\delta$ and is such that $\braket{\nabla\obj(\state_{\run})}{z} \geq \sharp\norm{z}/2$ for all $\run\geq \run_{0}$ (by the sharpness assumption).
Our assertion then follows by conditioning on this event and proceeding as in the proof of \cref{thm:sharp-global}.
\end{proof}

We close this section by noting that the convergence of $\act_{\run}$ in a finite number of steps is a unique feature of \emph{lazy} descent schemes with a surjective mirror map.
For example, if we consider the \emph{greedy} (or ``\emph{eager}'') projected descent scheme
\begin{equation}
\label{eq:SGD-greedy}
\act_{\run+1}
	= \Eucl(\act_{\run} - \step_{\run} \nabla\sobj(\act_{\run};\sample_{\run})),
\end{equation}
it is not possible to obtain a result similar to \cref{thm:sharp-global,thm:sharp-local} without further assumptions on the stochasticity affecting \eqref{eq:opt}.
To see why, assume that $\sol$ is a sharp minimum of \eqref{eq:opt} and $\act_{\run} = \sol$ for some $\run$.
If the sampled gradient $\nabla\sobj(\act_{\run};\sample_{\run})$ attains all directions with positive probability (more precisely, if the unit vector $\nabla\sobj(\act_{\run};\sample_{\run})/\dnorm{\nabla\sobj(\act_{\run};\sample_{\run})}$ is supported on the entire unit sphere $\sphere^{\vdim}$ of $\R^{\vdim}$), there exists some $\delta>0$ such that
\begin{equation}
\probof{\nabla\sobj(\act_{\run};\sample_{\run}) \notin \dcone(\sol)}
	\geq \delta
	\quad
	\text{for all $\run$}.
\end{equation}
We thus obtain
\begin{equation}
\probof{\act_{\run+1}\neq\sol \given \act_{\run}=\sol}
	\geq \delta
	\quad
	\text{for all $\run$},
\end{equation}
implying in turn that $\act_{\run}$ cannot converge to $\sol$ in a finite number of iterations.
We find this property of lazy descent schemes particularly appealing as it ensures very fast convergence in the presence of sharp minima.

\section{Numerical experiments}
\label{sec:numerics}

In this section, we validate the theoretical analysis of the previous sections via a series of numerical experiments.
As a first illustration of \cref{thm:global,thm:weak}, we begin by plotting the generated trajectories of \eqref{eq:SMD} for two non-convex test functions satisfying the coherence requirements of \cref{def:VC,def:VC-weak} respectively.
Referring to \cref{fig:portraits} for the detailed expressions, the specific setup is as in \cref{ex:noisy} with $\noise$ following a standard Gaussian distribution;
\eqref{eq:SMD} was then run with the Euclidean projector of \cref{ex:Eucl} and a step-size sequence $\step_{\run} \propto 1/\run$.
In both cases, the (randomly) generated trajectories of \eqref{eq:SMD} are seen to converge to a minimum point of \eqref{eq:opt}, even when the problem's minimum set is non-convex (as in the second example plotted in \cref{fig:portraits}).

To go beyond globally coherent problems, we also test the convergence of \eqref{eq:SMD} against the widely used Rosenbrock benchmark of \cref{eq:Rosenbrock}.
This test function admits a unique global minimum point at $\sol = (1,\dotsc,1)$;
however, this minimum is at the lowest point of a very flat and thin parabolic valley which is notoriously difficult for first-order methods to traverse \cite{Ros60}.
Because of the parabolic shape of this valley, the problem is not globally coherent (there are rays emanating from $\sol$ along which $\obj$ fails to be nondecreasing) but an easy algebraic calculation in the spirit of \cref{ex:star} shows that $\sol$ is locally coherent.


\begin{figure}[tbp]
\subfigure
[$\obj(r,\theta) = (3 + \sin(5\theta) + \cos(3\theta)) r^{2} (5/3 - r)$ in polar coordinates ($0\leq r\leq1$, $0\leq\theta\leq 2\pi$).]
{\includegraphics[width=.48\textwidth]{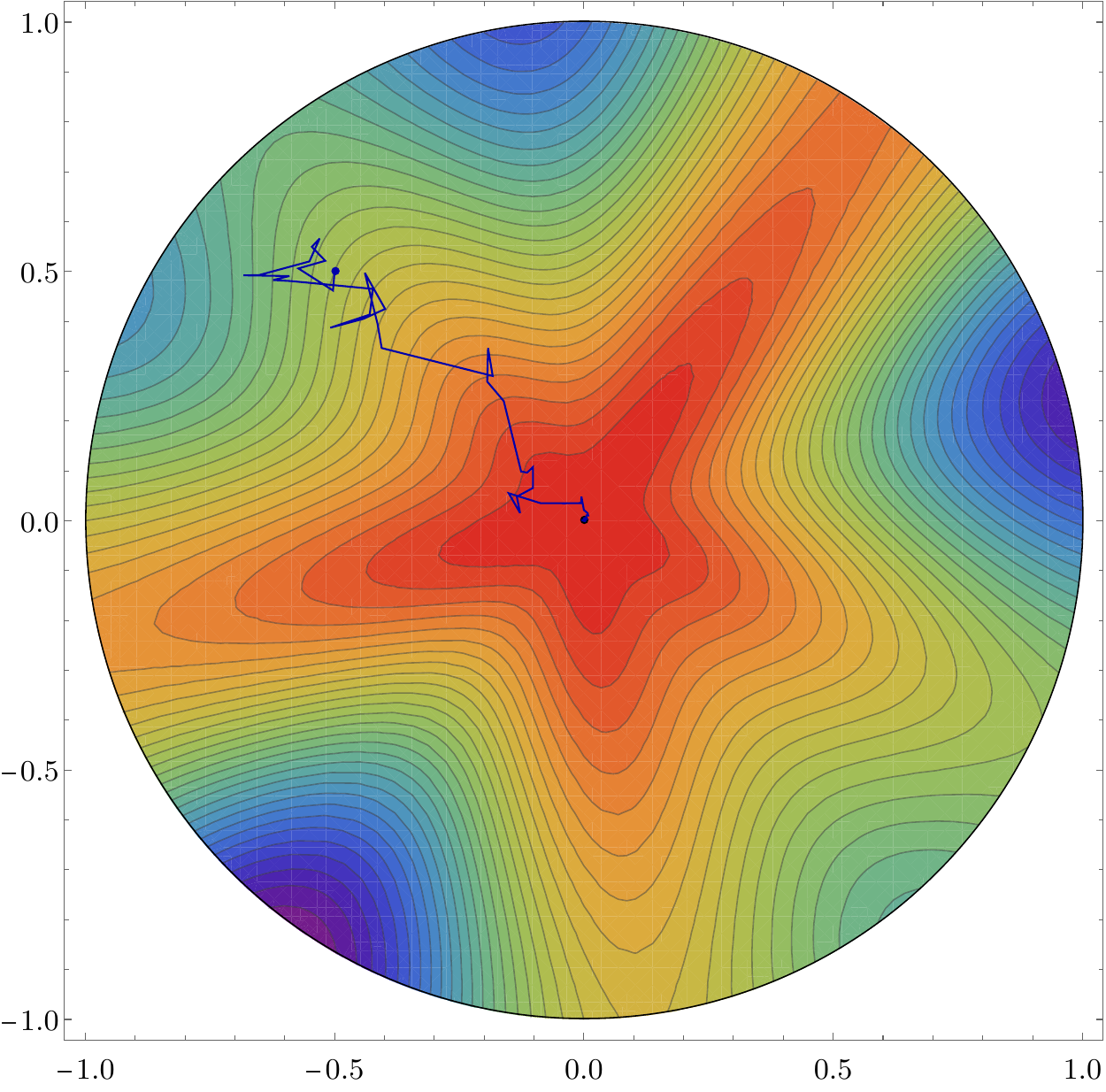}
\quad
\includegraphics[width=.48\textwidth]{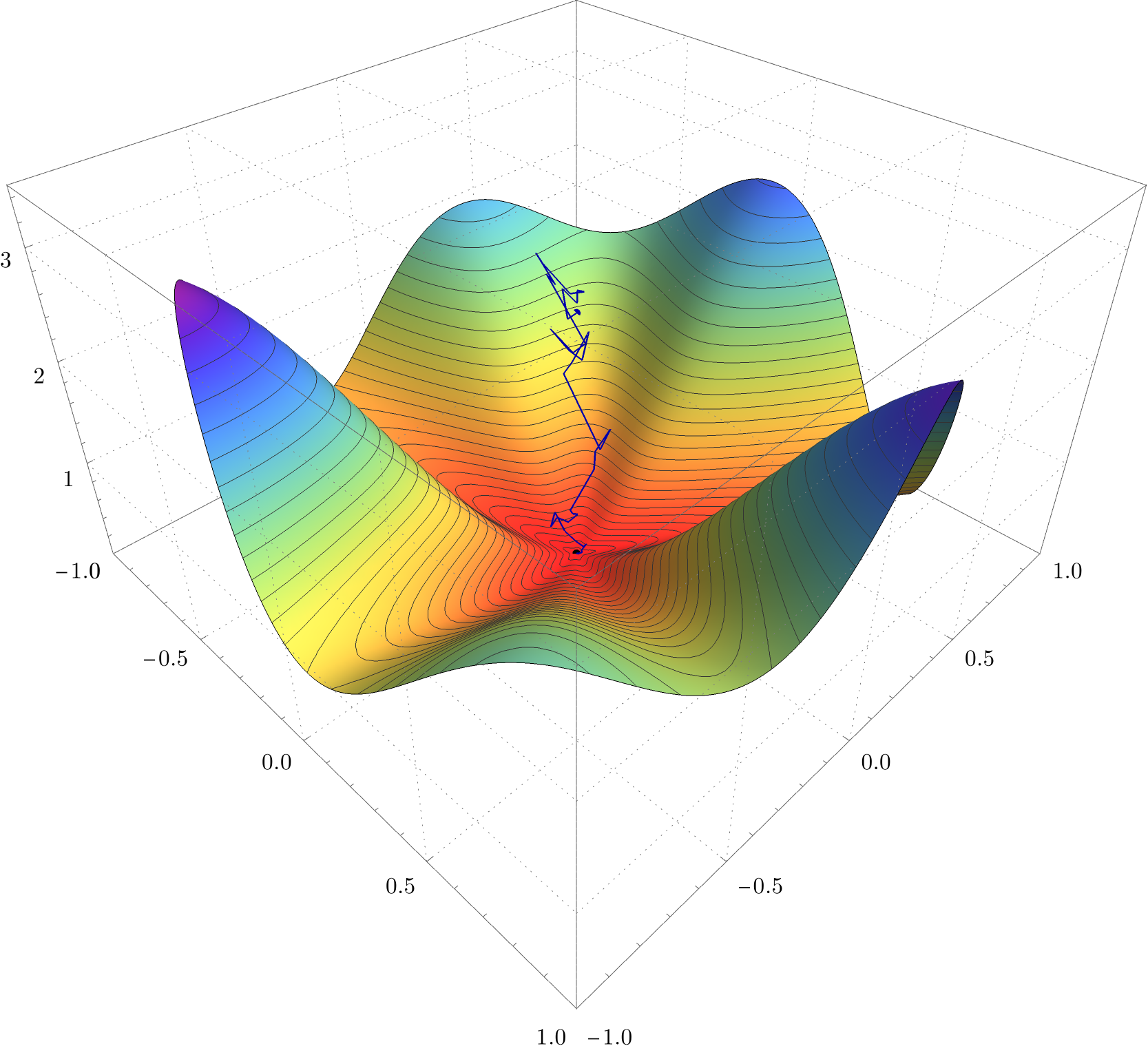}}
\\%
\subfigure
[$\obj(\state_{1},\state_{2}) = \state_{1}^{2} \state_{2}^{2}$, $-1\leq\state_{1},\state_{2} \leq 1$.]
{\includegraphics[width=.48\textwidth]{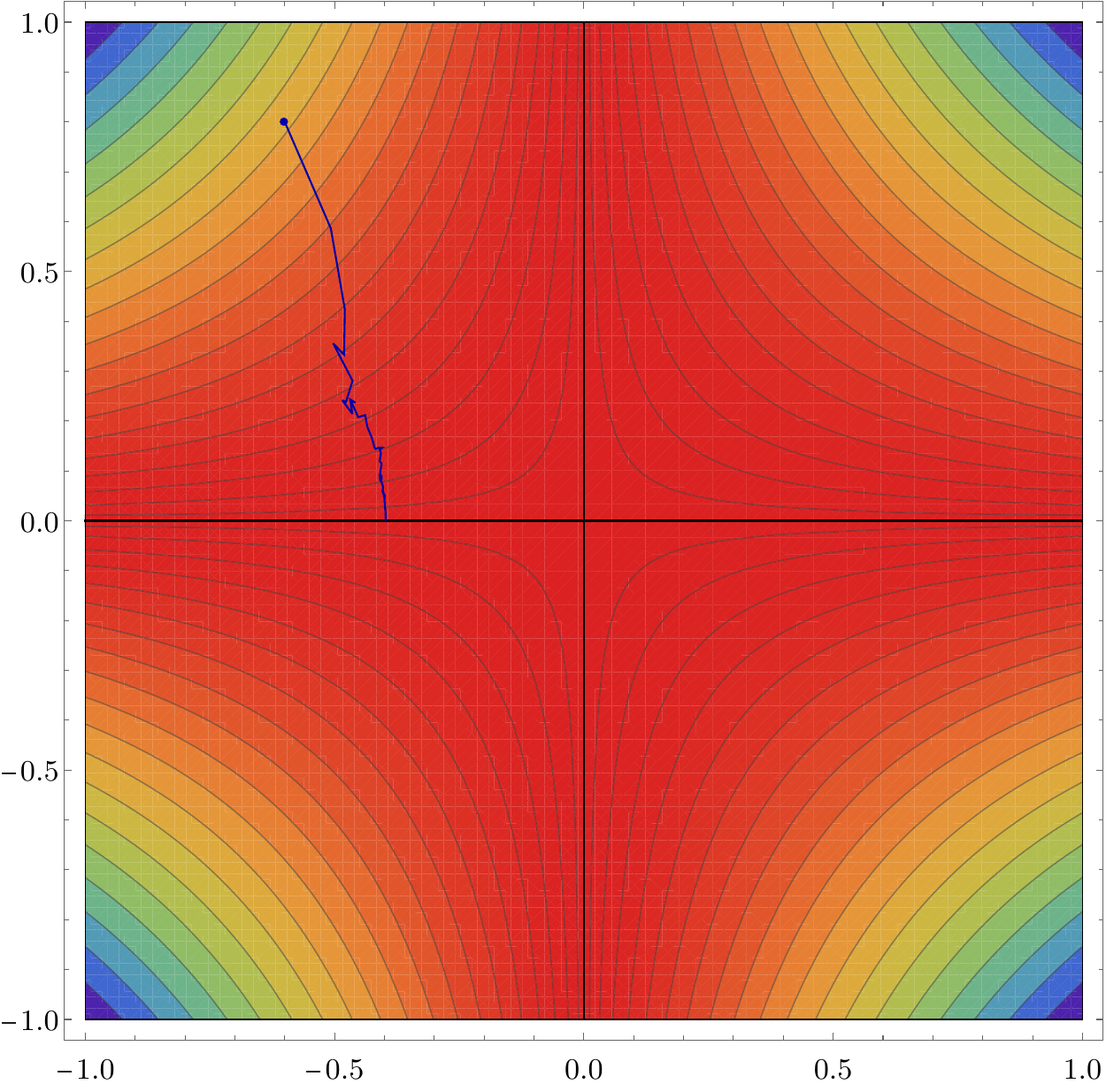}
\quad
\includegraphics[width=.48\textwidth]{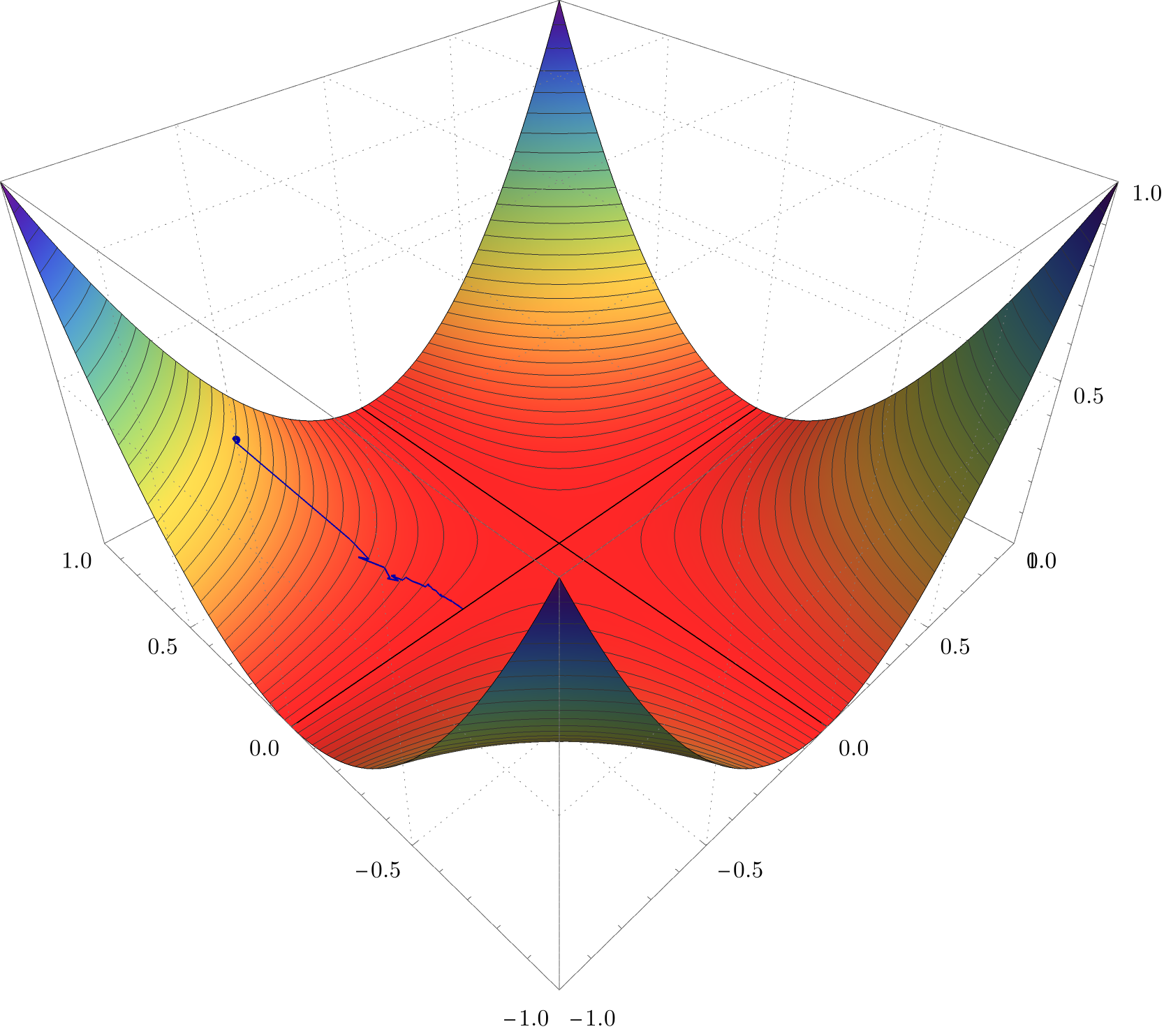}}%
\caption{Convergence of \ac{SMD} in a coherent problem with a unique minimizer (top) and a weakly coherent problem with a non-convex minimum set (bottom).
In both cases, the minimum set of $\obj$ is plotted in solid black.}%
\label{fig:portraits}
\end{figure}



\begin{figure}[tbp]
\includegraphics[width=.48\textwidth]{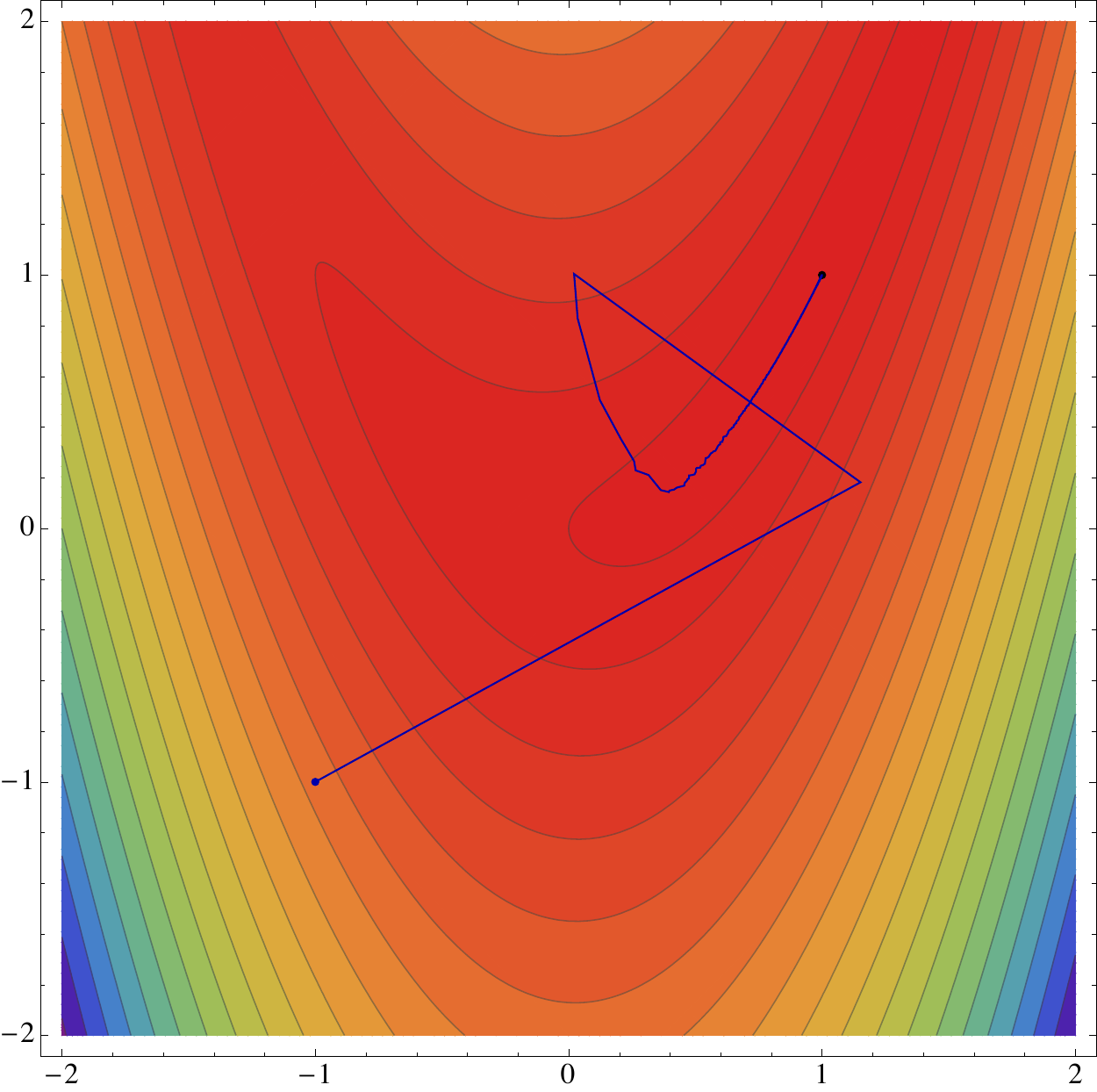}
\hfill
\includegraphics[width=.48\textwidth]{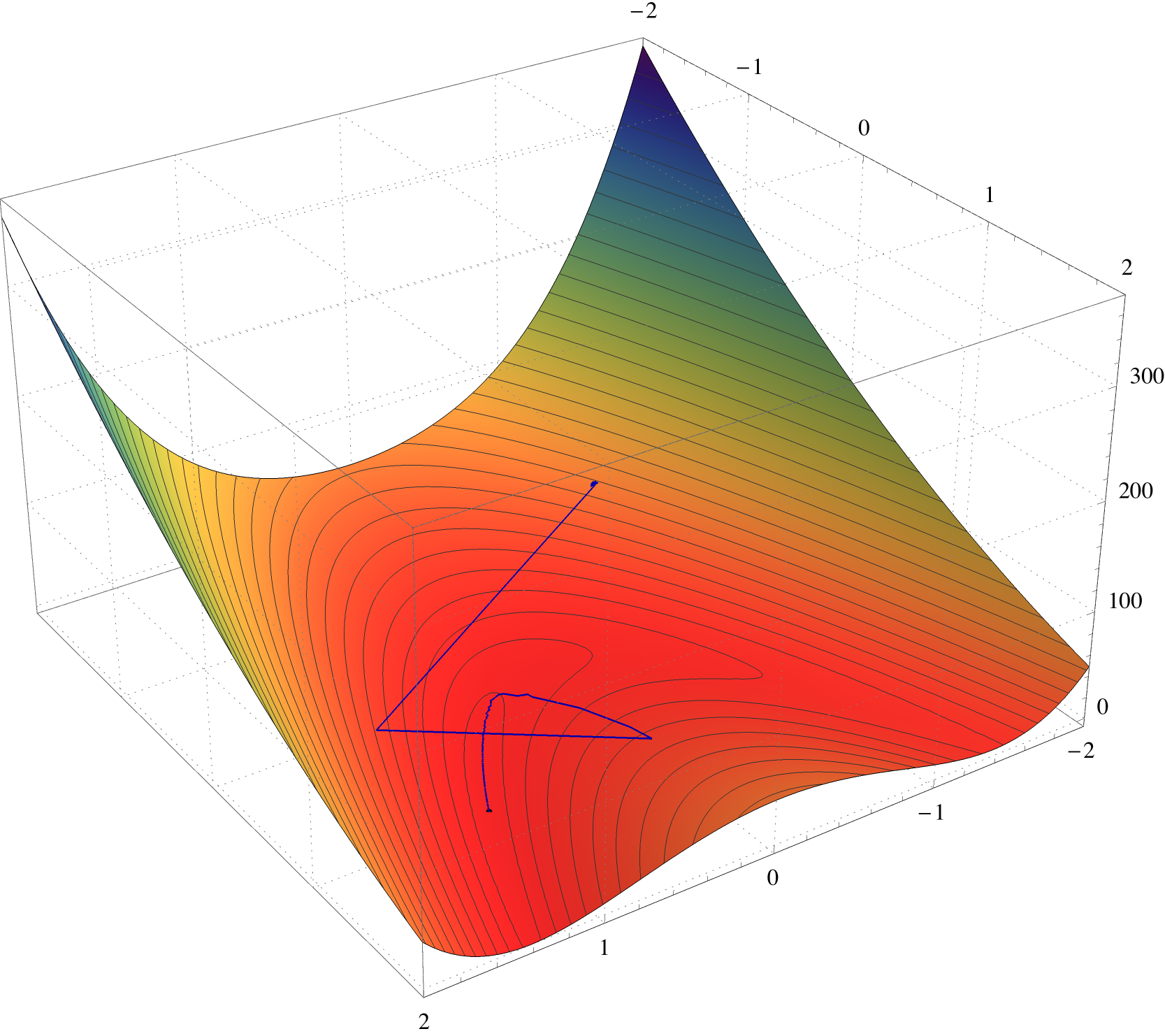}
\caption{Convergence of the \ac{SMD} algorithm in the Rosenbrock test with $\vdim=2$ degrees of freedom.}%
\label{fig:Rosenbrock}
\end{figure}


For illustration purposes, we first focus on a low-dimensional example with $\vdim=2$ degrees of freedom and algorithmic parameters as in \cref{fig:portraits}.
Despite the lack of global coherence, the simulated trajectories of \eqref{eq:SMD} quickly reach the Rosenbrock valley and eventually converge to the minimum of $\obj$;
a typical such trajectory is shown in \cref{fig:Rosenbrock}.
Subsequently, in \cref{fig:rates}, we run a series of tests on the Rosenbrock function for $\vdim=10^{3}$ and $\vdim=10^{4}$ degrees of freedom.
Because the calculation of the gradient becomes increasingly difficult as $\vdim$ grows large, we take the approach of \cref{ex:distributed} and, at each iteration $\run=\running$ of the algorithm, we randomly pick an integer between $1$ and $\vdim$ and calculate the gradient of
\(
\obj_{i}(\state)
	= 100 (\state_{i+1} - \state_{i})^{2} + (1 - \state_{i})^{2}.
\)


\begin{figure}[tbp]
\includegraphics[height=.315\textwidth]{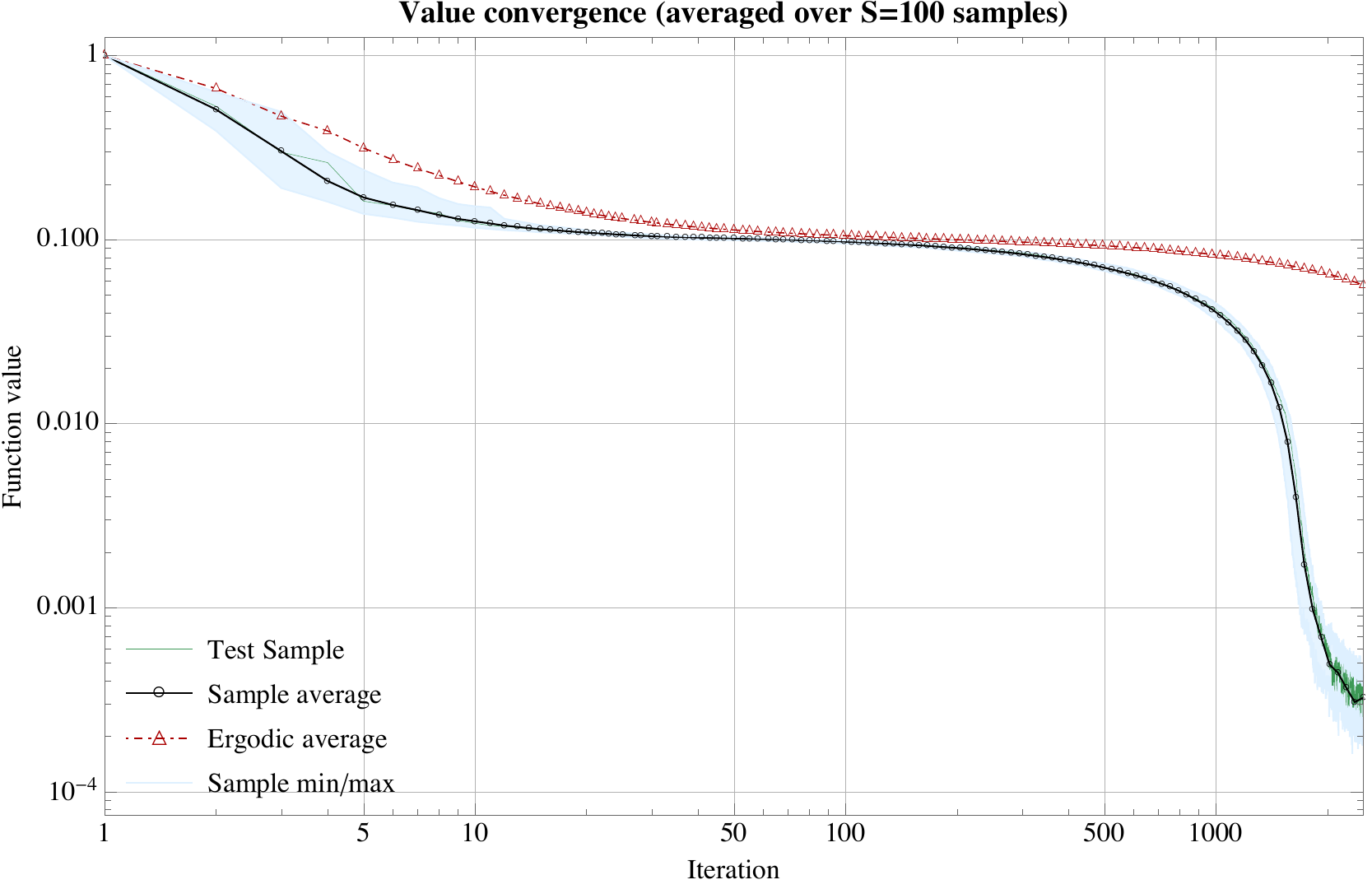}
\hfill
\includegraphics[height=.315\textwidth]{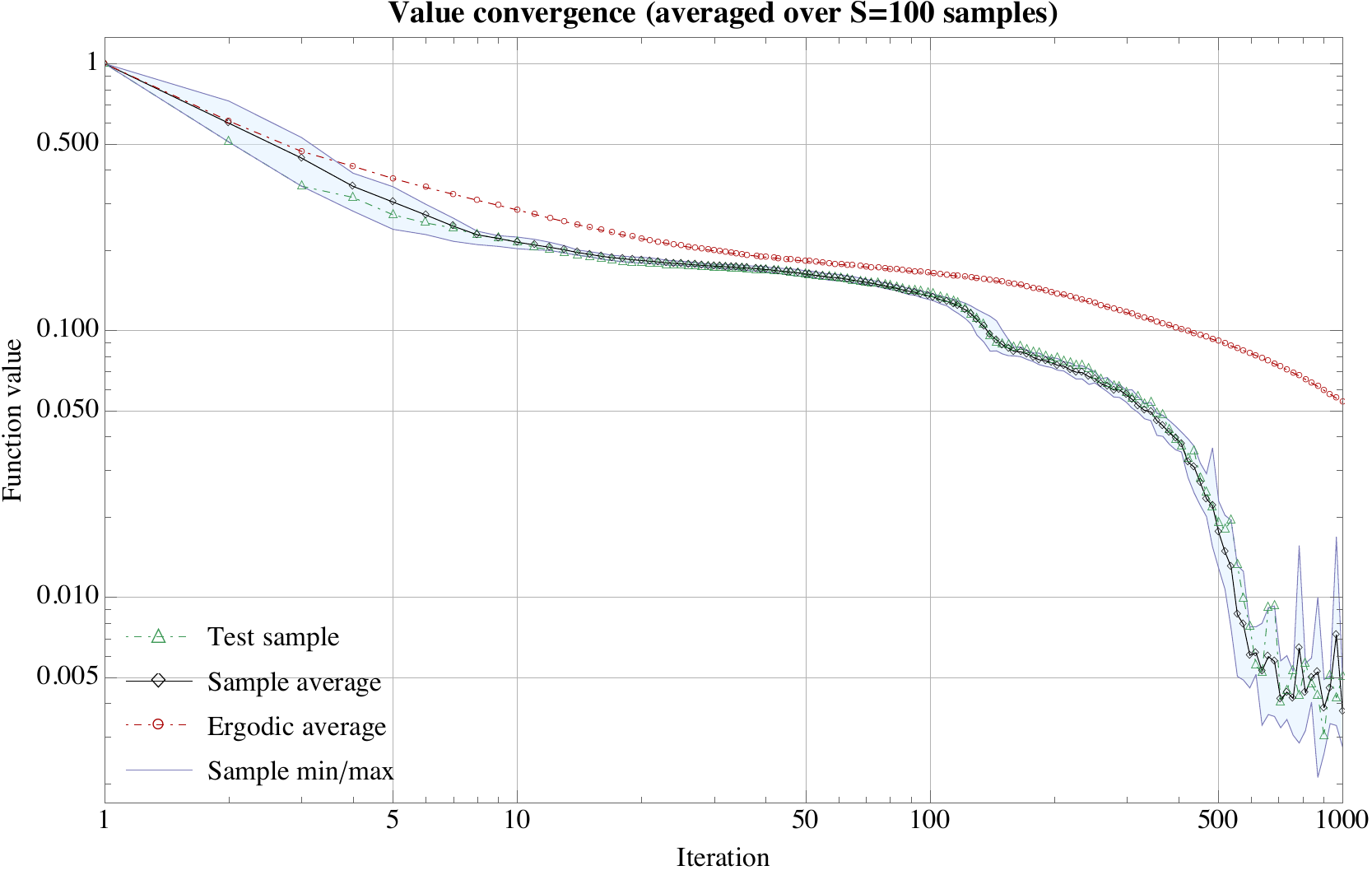}
\caption{Convergence speed of \ac{SMD} in the Rosenbrock benchmark for $\vdim=10^{3}$ and $\vdim=10^{4}$ degrees of freedom (left and right respectively).
The lightly shaded envelope indicates the best and worst realizations of the algorithm over $S=100$ sample runs;
the corresponding sample mean is represented by a solid black line.
For comparison purposes, we also plot the performance of the ergodic average $\bar\act_{\run} = \sum_{\runalt=1}^{\run} \step_{\runalt}\act_{\runalt} \big/ \sum_{\runalt=1}^{\run} \step_{\runalt}$ of $\act_{\run}$ (dashed red line).
Due to lack of convexity, the ergodic average of $\act_{\run}$ converges at a significantly slower rate.}%
\label{fig:rates}
\end{figure}


In so doing, we obtain the plots of \cref{fig:rates} where, for statistical significance, we report the findings of $S=100$ sample realizations.
For comparison purposes, we also include in the figure the performance of the ergodic average $\bar\act_{\run}$ of $\act_{\run}$ as defined in \cref{eq:ergodic}.
This sequence is the standard output of \acl{MD}/\acl{DA} schemes in convex problems;
however, in our non-convex setting, this averaging offers no tangible benefits.
This is seen clearly in \cref{fig:rates} where the convergence speed of $\bar\act_{\run}$ is considerably slower than that of the algorithm's last iterate.

Finally, in \cref{fig:concave}, we examine the convergence rate of \eqref{eq:SMD} for quadratic objective functions of the form
\begin{equation}
\label{eq:quadratic}
\obj(\state)
	= \frac{1}{2} \sum_{i,j=1}^{\vdim} Q_{ij} \state_{i}\state_{j}
	+ \sum_{i=1}^{\vdim} b_{i}\state_{i},
\end{equation}
with $Q = (Q_{ij})_{i,j=1}^{\vdim}$ negative-definite (so $\obj$ is concave).
When $\state$ is constrained to lie on the unit simplex of $\R^{\vdim}$, the minimization of $\obj$ is related to the maximum weight clique problem \cite{MPB02}:
this problem is well known to be NP-hard, so fast convergence to local minima of $\obj$ is essential in order to get reasonable approximate solutions.


\begin{figure}[tbp]
\includegraphics[width=.49\textwidth]{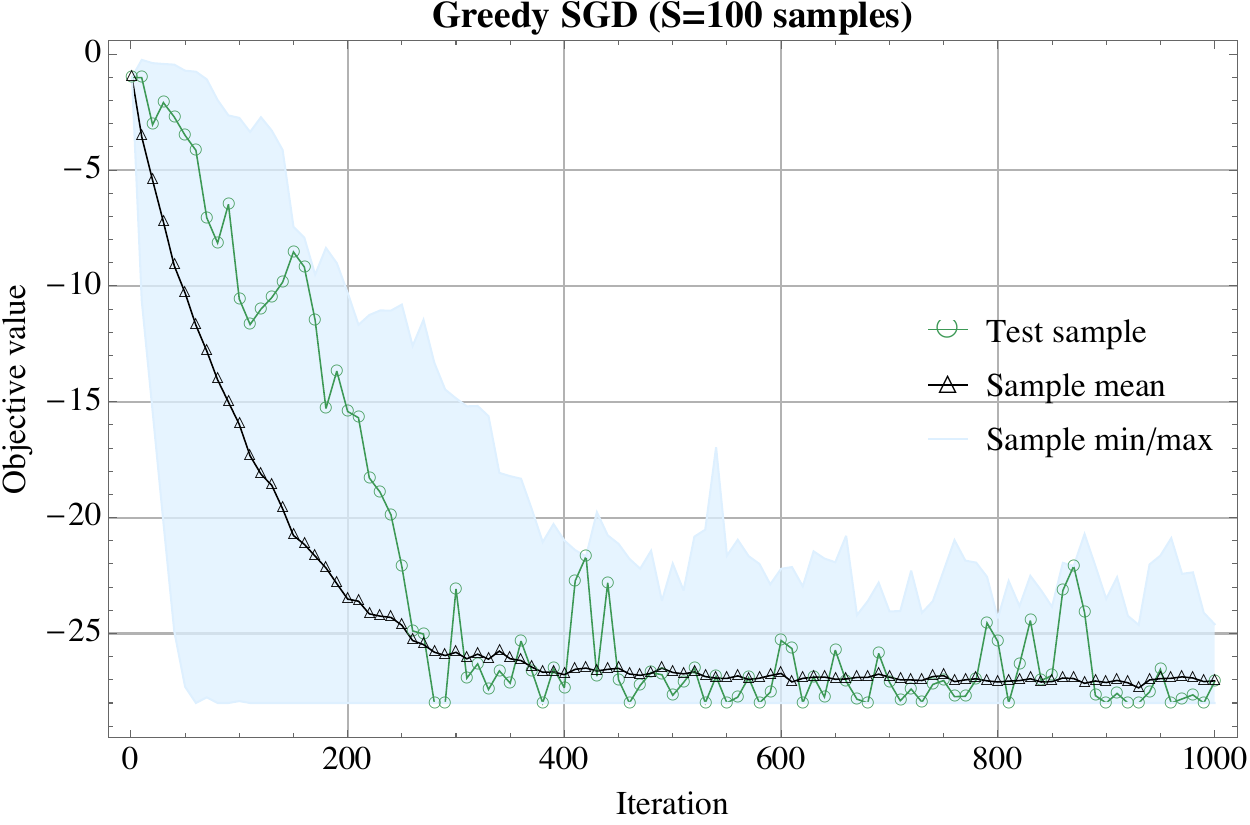}
\hfill
\includegraphics[width=.49\textwidth]{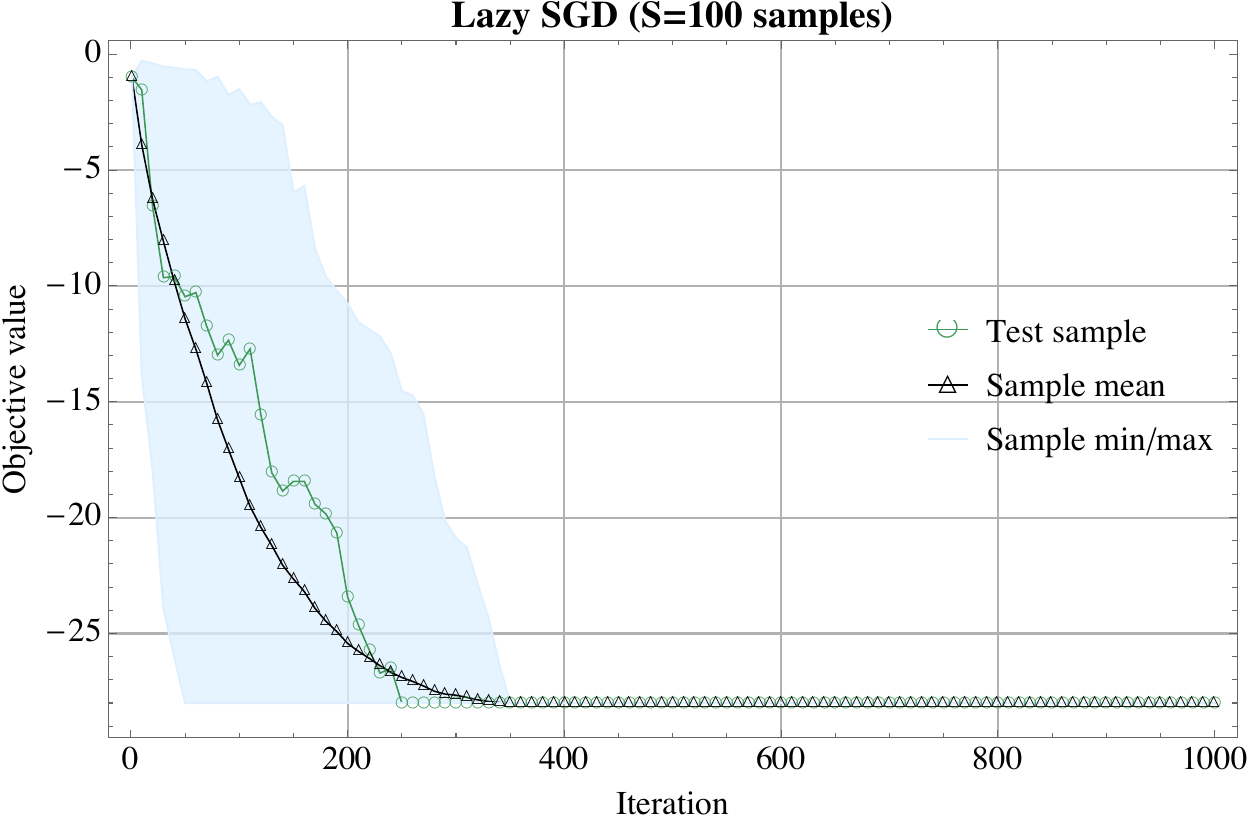}
\caption{Convergence of the lazy and greedy variants of \ac{SGD} in a quadratic minimization problem of the form \eqref{eq:quadratic} with $\vdim=100$.
The lightly shaded area traces the best and worst realizations of the algorithm over $S=100$ sample runs;
the corresponding sample mean is drawn as a solid black line.
In the dedicated runtime ($\run=1000$ iterations), the greedy variant still hasn't converged;
by contrast, even the worst realization of lazy \ac{SGD} has converged within approximately $300$ iterations.}%
\label{fig:concave}
\end{figure}


Using again the stochastic setup of \cref{ex:noisy}, we ran both \cref{alg:SGD} and its greedy variant \eqref{eq:SGD-greedy} for a concave quadratic objective of the form \eqref{eq:quadratic} with $\vdim=100$ and randomly drawn $Q$ and $b$.
As can be seen in \cref{fig:concave}, the lazy variant of \ac{SGD} converges within a finite number of iterations whereas the greedy variant still oscillates within the allocated time window.
This behavior is explained by \cref{thm:sharp-local} and the discussion that follows:
because the greedy variant essentially ``remembers'' only the last state, convergence within a finite number of iterations is not possible;
by contrast, the dual averaging that takes place in the lazy variant allows $\act_{\run}$ to converge in finite time, despite all the noise.

\section{Discussion}
\label{sec:discussion}

To conclude, we first note that our analysis can be extended to the study of stochastic variational inequalities with possibly non-monotone operators.
The notion of a variationally coherent problem still make sense for a variational inequality ``as is'', and the Fenchel coupling can also be used to establish almost sure convergence to the solution set of a variational inequality.
That said, there are several details that need to be adjusted along the way, so we leave this direction to future work.

Finally, we should also mention that another merit of \ac{SMD} is that, at least for (strongly) convex optimization problems~\cite{duchi2015asynchronous,recht2011hogwild}, the algorithm is amenable to asynchronous parallelization.
This is an increasingly desirable advantage, especially in the presence of large-scale datasets that are characteristic of ``big data'' applications requiring the computing power of a massive number of parallel processors. 
Although we do not tackle this question in this paper, the techniques developed here can potentially be leveraged to provide theoretical guarantees for certain non-convex stochastic programs when \ac{SMD} is run asynchronously and in parallel.

\appendix

\section{Elements of martingale limit theory}
\label{app:aux}
%
%
In this appendix, we state for completeness some basic results from martingale limit theory which we use throughout our paper.
The statements are adapted from \cite{HH80} where we refer the reader for detailed proofs.

We begin with a strong \acl{LLN} for \aclp{MDS}:

\begin{theorem}
\label{thm:mg_convergence1}
Let $M_{\run} = \sum_{\runalt=\start}^{\run} d_{\runalt}$ be a martingale with respect to an underlying stochastic basis $(\Omega,\filter,(\filter_{\run})_{\run=\start}^{\infty},\prob)$ and let $(\tau_{\run})_{\run=\start}^{\infty}$ be a nondecreasing sequence of positive numbers with $\lim_{\run\to\infty} \tau_{\run} = \infty$.
If $\sum_{\run=\start}^{\infty} \tau_{\run}^{-p} \exof{\abs{d_{\run}}^{p}\given\filter_{\run-1}} <\infty$ for some $p\in[1,2]$ \as, we have:
\begin{equation}
\label{eq:LLN}
\lim_{\run\to\infty} \frac{M_{\run}}{\tau_{\run}} = 0
	\quad
	\as
\end{equation}
\end{theorem}

The second result we use is Doob's martingale convergence theorem:

\begin{theorem}
\label{thm:mg_convergence2}
If $M_{\run}$ is a submartingale that is bounded in $L^{1}$ \textpar{i.e. $\sup_{\run} \exof{\abs{M_{\run}}} < \infty$}, $M_{\run}$ converges almost surely to a random variable $M$ with $\exof{\abs{M}} < \infty$.
\end{theorem}

The next result is also due to Doob, and is known as Doob's maximal inequality:

\begin{theorem}
\label{thm:mg_convergence3}
Let $M_{\run}$ be a non-negative submartingale and fix some $\eps>0$.
Then:
\begin{equation}
\probof{\sup\nolimits_{\run} M_{\run} \geq \epsilon}
	\leq \frac{\exof{M_{\run}}}{\epsilon}.
\end{equation}
\end{theorem}

Finally, a widely used variant of Doob's maximal inequality is the following:

\begin{theorem}
\label{thm:mg_convergence4}
With assumptions and notation as above, we have
\begin{equation}
\probof{\sup\nolimits_{\run} \abs{M_{\run}} \geq \epsilon}
	\leq \frac{\exof{M_{\run}^{2}}}{\epsilon^2}.
\end{equation}
\end{theorem}

\section{Technical proofs}
\label{app:proofs}
%
%
In this appendix, we present the proofs that were omitted from the main text.
We begin with the core properties of the Fenchel coupling:

\begin{proofof}{Proof of \cref{lem:Fenchel}}
To prove the first claim, let $\state =\mirror(\dstate) = \argmax_{\alt\state\in\states}\{\braket{\dstate}{\alt\state} - \hreg(\alt\state)\}$, so $\dstate\in\pd \hreg(\state)$ from standard results in convex analysis \cite{Roc70}.
We thus get:
\begin{equation}
\label{eq:Fench-sub}
\fench(\base,\dstate) = \hreg(\base) + \hreg^{\ast}(\dstate) - \braket{\dstate}{\base}
	= \hreg(\base) - \hreg(\state) - \braket{\dstate}{\base - \state}.
\end{equation}
Since $\dstate \in \pd \hreg(\state)$ and $\hreg$ is $\strong$-strongly convex, we also have
\begin{equation}
\hreg(\state) + \tau\braket{\dstate}{\base - \state}
	\leq \hreg(\state + \tau(\base - \state))
	\leq (1-\tau) \hreg(\state) + \tau \hreg(\base) - \tfrac{1}{2} \strong \tau(1-\tau) \norm{\state - \base}^{2}
\end{equation}
for all $\tau\in[0,1]$, thereby leading to the bound
\begin{equation}
\label{eq:divbound}
\tfrac{1}{2} \strong(1-\tau) \norm{\state - \base}^{2}
	\leq \hreg(\base) - \hreg(\state) - \braket{\dstate}{\base - \state}
	= \fench(\base,\dstate).
\end{equation}
Our claim then follows by letting $\tau\to0^{+}$ in \eqref{eq:divbound}.

For our second claim, we start by citing a well-known duality principle for strongly convex functions \cite[Theorem 12.60]{RW98}:
If $f\from\R^{\vdim}\to\R\cup\{-\infty,+\infty\}$ is proper, lower semi-continuous and convex, its convex conjugate $f^{\ast}$ is $\sigma$-strongly convex if and only if $f$ is differentiable and satisfies
\begin{equation}
f(\alt\state)
	\leq f(\state) + \braket{\nabla f(\state)}{\alt\state - \state} + \frac{1}{2\sigma} \norm{\alt\state - \state}^{2}
\end{equation}
for all $\state,\alt\state\in\R^{\vdim}$.
In our case, if we extend $\hreg$ to all of $\vecspace$ by setting $\hreg\equiv+\infty$ outside $\states$, we have that $\hreg$ is $\strong$-strongly convex, lower semi-continuous and proper, so $(\hreg^{\ast})^{\ast} = \hreg$ \cite[Theorem 11.1]{RW98}.
It is also easy to see that $\hreg^{\ast}$ is proper, lower semi-continuous and convex (since it is a point-wise maximum of affine functions by definition), so the $\strong$-strong convexity of $\hreg = (\hreg^{\ast})^{\ast}$ implies that $\hreg^{\ast}$ is differentiable and satisfies
\begin{flalign}
\label{eq:smoothness}
\hreg^{\ast}(\alt\dstate)
	&\leq \hreg^{\ast}(\dstate)
	+ \braket{\alt\dstate - \dstate}{\nabla \hreg^{\ast}(\dstate)}
	+ \frac{1}{2\strong} \dnorm{\alt\dstate - \dstate}^{2}
	\\
	&= \hreg^{\ast}(\dstate)
	+ \braket{\alt\dstate - \dstate}{\mirror(\dstate)} +
	\frac{1}{2\strong} \dnorm{\alt\dstate - \dstate}^{2}
\end{flalign}
for all $\dstate,\alt\dstate\in\dstates$, where the last equality follows from the fact that $\nabla \hreg^{\ast}(\dstate) = \mirror(\dstate)$.
Therefore, substituting the preceding inequality in the definition of the Fenchel coupling, we obtain
\begin{flalign}
\fench(\state,\alt\dstate)
	&= \hreg(\state) + \hreg^{\ast}(\alt\dstate) - \braket{\alt\dstate}{\state}
	\notag\\
	&\leq \hreg(\state) + \hreg^{\ast}(\dstate) + \braket{\alt\dstate - \dstate}{\nabla \hreg^{\ast}(\dstate)} + \frac{1}{2\strong} \dnorm{\alt\dstate - \dstate}^{2} - \braket{\alt\dstate}{\state}
	\notag\\
	&= \fench(\state,\dstate) + \braket{\alt\dstate - \dstate}{\mirror(\dstate) - \state} + \frac{1}{2\strong} \dnorm{\alt\dstate - \dstate}^{2},
\end{flalign}
and our assertion follows.
\end{proofof}

We now turn to the recurrence properties of \ac{SMD}:

\begin{proofof}{Proof of \cref{prop:recurrence}}
Our proof proceeds step-by-step, as discussed in \cref{sec:recurrence}:

\setcounter{proofstep}{0}
\begin{proofstep}{Martingale properties of $\score_{\run}$}
By \cref{asm:L2} and the fact that finiteness of second moments implies finiteness of first moments, we get $\exof{\dnorm{\sobj(\state;\sample_{\run})}} < \infty$.
We then claim that $\noise_{\run} = \nabla\obj(\act_{\run}) - \nabla\sobj(\act_{\run};\sample_{\run})$ is an $L^{2}$-bounded \acl{MDS} with respect to the natural filtration of $\sample_{\run}$.
Indeed, we have:
\begin{enumerate}
\addtolength{\itemsep}{1ex}
\item
Since $\act_{\run}$ is $\filter_{\run-1}$-measurable and $\sample_{\run}$ is \ac{iid}, we readily get
\begin{flalign}
\exof{\noise_{\run} \given \filter_{\run-1}}
	&= \exof{\nabla\obj(\act_{\run}) - \nabla\sobj(\act_{\run};\sample_{\run}) \given \filter_{\run-1}}
	\notag\\
	&= \exof{\nabla\obj(\act_{\run}) - \nabla\sobj(\act_{\run};\sample_{\run}) \given \sample_{1},\dotsc,\sample_{\run-1}}
	\notag\\
	&= \nabla\obj(\act_{\run}) - \nabla\obj(\act_{\run})
	\notag\\
	&= 0.
\end{flalign}

\item
Furthermore, by \cref{asm:L2}, the $L^{2}$ norm of $\noise$ satisfies
\begin{flalign}
\label{eq:vbound}
\exof{\dnorm{\noise_{\run}}^{2} \given \filter_{\run-1}}
	&= \exof{\dnorm{\nabla\obj(\act_{\run}) - \nabla\sobj(\act_{\run};\sample_{\run})}^{2} \given \filter_{\run-1}}
	\notag\\
	&\leq 2\exof{\dnorm{\nabla\obj(\act_{\run})}^{2} \given \filter_{\run-1}}
	+ 2 \exof{\dnorm{\nabla\sobj(\act_{\run};\sample_{\run})}^{2} \given \filter_{\run-1}}
	\notag\\
	&\leq 2 \dnorm{\nabla\obj(\act_{\run})}^{2}
	+ 2 \vbound^{2}
	\notag\\
	&= 2\dnorm{\exof{\nabla\sobj(\act_{\run};\sample)}}^{2}
	+ 2\vbound^{2}
	\notag\\
	&\leq 2 \exof{\dnorm{\nabla\sobj(\act_{\run};\sample)}^{2}}
	+ 2\vbound^{2}
	\notag\\
	&\leq \noisevar,
\end{flalign}
where we set $\noisevar = 4\vbound^{2}$, and we used the dominated convergence theorem to interchange expectation and differentiation in the second line, and Jensen's inequality in the penultimate one.
\end{enumerate}
\end{proofstep}

\begin{proofstep}{Recurrence of $\eps$-neighborhoods}
We proceed to show that every $\eps$-neigh\-bor\-hood $\ball(\sols,\eps)$ of $\sols$ is recurrent under $\act_{\run}$.
To do so, fix some $\eps>0$ and assume ad absurdum that $\act_{\run}$ enters $\ball(\sols,\eps)$ only a finite number of times, so there exists some finite $\run_{0}$ such that $\dist(\sols,\act_{\run}) \geq \eps$ for all $\run\geq\run_{0}$.
Since $\states\setminus\ball(\sols,\eps)$ is compact, $\payv(\state) = -\nabla \obj(\state)$ is continuous in $\state$;
furthermore, letting $\sol$ be such that $\braket{\nabla\obj(\state)}{\state - \sol} = 0$ only if $\state\in\sols$ (cf.~\cref{def:VC}),  it follows that there exists some $\const \equiv \const(\eps) > 0$ such that
\begin{equation}
\braket{\payv(\state)}{\state - \sol}
	\leq - \const
	< 0
	\quad
	\text{for all $\state\in\states\setminus\ball(\sols,\eps)$}.
\end{equation}
To proceed, let $R = \max_{\state \in \states} \norm{\state}$ and set $\beta_{\run} = \step_{\run}^{2}/(2\strong)$.
Then, letting $\fench_{\run} = \fench(\sol,\score_{\run})$ and $\snoise_{\run} = \braket{\noise_{\run}}{\act_{\run}-\sol}$, \cref{lem:Fenchel} yields
\begin{flalign}
\fench_{\run+1}
	= \fench(\sol,\score_{\run+1})
	&= \fench(\sol, \score_{\run} + \step_{\run} \est\payv_{\run})
	\notag\\
	&\leq \fench(\sol, \score_{\run})
	+ \step_{\run} \braket{\payv(\act_{\run}) + \noise_{\run}}{\act_{\run} - \sol}
	+ \beta_{\run} \dnorm{\est\payv_{\run}}^{2}
	\notag\\
	&=\fench_{\run}
	+ \step_{\run} \braket{\payv(\act_{\run})}{\act_{\run} - \sol}
	+ \step_{\run} \snoise_{\run}
	+ \beta_{\run} \dnorm{\est\payv_{\run}}^{2}
	\notag\\
	&\leq \fench_{\run}
	- \step_{\run} \const
	+ \step_{\run} \snoise_{\run}
	+ \beta_{\run} \dnorm{\est\payv_{\run}}^{2}.
\end{flalign}
Hence, letting $\tau_{\run} = \sum_{\runalt=\run_{0}}^{\run} \step_{\runalt}$ and telescoping from $\run_{0}$ to $\run$, we get
\begin{flalign}
\label{eq:bound-Fench}
\fench_{\run+1}
	&\leq \fench_{\run_{0}}
	- \const \sum_{\runalt=\run_{0}}^{\run} \step_{\runalt}
	+ \sum_{\runalt=\run_{0}}^{\run} \step_{\runalt} \snoise_{\runalt}
	+ \sum_{\runalt=\run_{0}}^{\run} \beta_{\runalt} \dnorm{\est\payv_{\runalt}}^{2}
	\notag\\
	&= \fench_{\run_{0}}
	- \tau_{\run} \bracks*{\const - \frac{\sum_{\runalt=\run_{0}}^{\run} \step_{\runalt}\snoise_{\runalt}}{\tau_{\run}}}
	+ \sum_{\runalt=\run_{0}}^{\run} \beta_{\runalt} \dnorm{\est\payv_{\runalt}}^{2}
\end{flalign}

We now proceed to bound each term of \eqref{eq:bound-Fench}.
First, by construction, we have
\begin{equation}
\exof{\snoise_{\run} \given \filter_{\run-1}}
	= \exof{\braket{\noise_{\run}}{\act_{\run} - \sol} \given \filter_{\run-1}}
	= \braket{\exof{\noise_{\run} \given \filter_{\run-1}}}{\act_{\run} - \sol}
	= 0,
\end{equation}
where we used the fact that $\act_{\run}$ is $\filter_{\run-1}$-measurable.
Moreover, by Young's inequality, we also have
\begin{equation}
\abs{\snoise_{\run}}
	= \abs{\braket{\noise_{\run}}{\act_{\run} - \sol}}
	\leq \dnorm{\noise_{\run}} \norm{\act_{\run} - \sol}
	\leq 2R\dnorm{\noise_{\run}},
\end{equation}
where, as before, $R = \max_{\state\in\states} \norm{\state}$.
\cref{eq:vbound} then gives
\begin{equation}
\exof{\snoise_{\run}^{2} \given \filter_{\run-1}}
	\leq (2R)^{2} \exof{\dnorm{\noise_{\run}}^{2} \given \filter_{\run-1}}
	\leq 4 R^{2} \noisevar,
\end{equation}
implying in turn that $\snoise_{\run}$ is an $L^{2}$-bounded \acl{MDS}.
It then follows that $\snoise_{\run}$ satisfies the summability condition
\begin{equation}
\sum_{\run=\run_{0}}^{\infty} \frac{\exof{\abs{\step_{\run} \snoise_{\run}}^{2} \given \filter_{\run-1}}}{\tau_{\run}^{2}}
	\leq 4R^{2}\noisevar \sum_{\run=\run_{0}}^{\infty} \frac{\step_{\run}^{2}}{\tau_{\run}^{2}}
	< \infty,
\end{equation}
where the last inequality follows from the assumption that $\step_{\run}$ is square-summable.
Thus, by the \acl{LLN} for \aclp{MDS} (\cref{thm:mg_convergence1}), we get
\begin{equation}
\frac{\sum_{\runalt=\run_{0}}^{\run} \step_{\runalt} \snoise_{\runalt}}{\tau_{\run}}
	\to 0
	\quad
	\text{as $\run\to\infty$ \as},
\end{equation}
and, with $\sum_{\runalt=\run_{0}}^{\infty} \step_{\runalt} = \infty$, we finally obtain
\begin{equation}
\lim_{\run\to\infty} \tau_{\run} \bracks*{\const - \frac{\sum_{\runalt=\run_{0}}^{\run} \step_{\runalt}\snoise_{\runalt}}{\tau_{\run}}}
	= \infty
	\quad
	\as.
\end{equation}

For the last term of \eqref{eq:bound-Fench}, let $S_{\run} = \sum_{\runalt=\run_{0}}^{\run} \beta_{\runalt} \dnorm{\est\payv_{\runalt}}^{2}$, so $S_{\run}$ is $\filter_{\run}$-measurable and nondecrasing.
In addition, we have
\begin{equation}
\exof{S_{\run}}
	= \sum_{\runalt=\run_{0}}^{\run} \beta_{\runalt} \exof{\dnorm{\est\payv_{\runalt}}^{2}}
	\leq \vbound^{2} \sum_{\runalt=\run_{0}}^{\run} \beta_{\runalt}
	< \infty,
\end{equation}
with the last step following from \eqref{eq:vbound}.
This implies that $S_{\run}$ is an $L^{1}$-bounded submartingale so, by Doob's submartingale convergence theorem (\cref{thm:mg_convergence2}), $S_{\run}$ converges almost surely to some random variable $S_{\infty}$, \ie the last term of \eqref{eq:bound-Fench} is bounded.
Hence, combining all of the above, we finally obtain
\begin{equation}
\limsup_{\run\to\infty} \fench_{\run}
	= - \infty
	\quad
	\as,
\end{equation}
contradicting the positive-definiteness of the Fenchel coupling (cf.~\cref{lem:Fenchel}).
We thus conclude that $\act_{\run}$ enters $\ball(\sols,\eps)$ infinitely many times \as, as claimed.
\end{proofstep}

\begin{proofstep}{Recurrence of Fenchel zones}
Using the reciprocity of the Fenchel coupling (\cref{asm:reciprocity}), we show below that every Fenchel zone $\fball(\sols,\delta)$ of $\sols$ contains an $\eps$-neighborhood of $\sols$.
Then, given that $\act_{\run}$ enters $\ball(\sols,\eps)$ infinitely often (per the previous step), it will also enter $\fball(\sols, \delta)$ infinitely often.

To establish this claim, assume instead that there is no $\eps$-ball $\ball(\sols,\eps)$ contained in $\fball(\sols, \delta)$.
Then, for all $\runalt>0$ there exists some $\dstate_{\runalt}\in\dstates$ such that $\dist(\sols,\mirror(\dstate_{\runalt})) = 1/\runalt$ but $\fench(\sols, y_{\delta}) \geq \eps$.
This produces a sequence $(\dstate_{\runalt})_{\runalt=1}^{\infty}$ such that $\dist(\sols,\mirror(\dstate_{\runalt})) \to 0$ but $\fench(\sols, \dstate_{\runalt}) \geq \eps$.
Since $\states$ is compact and $\sols$ is closed, we can assume without loss of generality that $\mirror(\dstate_{\runalt})\to\base$ for some $\base\in\sols$ (at worst, we only need to descend to a subsequence of $\dstate_{\runalt}$).
We thus get
\begin{equation}
\eps
	\leq \fench(\sols,\dstate_{\runalt})
	\leq \fench(\base,\dstate_{\runalt}).
\end{equation}
However, since $\mirror(\dstate_{\runalt})\to\base$, \cref{asm:reciprocity} gives $\fench(\base,\dstate_{\runalt})\to0$, a contradiction.
We conclude that $\fball(\sols,\delta)$ contains an $\eps$-neighborhood of $\sols$, as required.
\end{proofstep}
\end{proofof}

\bibliographystyle{siam}
\bibliography{IEEEabrv,openloop,Bibliography-SMD}

\end{document}